\documentclass[12pt]{amsart}
\usepackage{amssymb}
\usepackage{url}
\usepackage{graphicx}
\usepackage{tikz-cd}
\usepackage{mathtools}
\usepackage{upgreek}
\usepackage{comment}

\usepackage[all]{xy}

\allowdisplaybreaks

\newtheorem{theorem}[equation]{Theorem}

\newtheorem{lemma}[equation]{Lemma}
\newtheorem{proposition}[equation]{Proposition}
\newtheorem{corollary}[equation]{Corollary}

\newtheorem*{theorem:derhamisomorphism}{Theorem~\ref{T:derhamisomorphism}}
\newtheorem*{theorem:characterization}{Theorem~\ref{T:characterization}}

\theoremstyle{definition}

\newtheorem{example}[equation]{Example}

\newtheorem{remark}[equation]{Remark}

\numberwithin{equation}{section}

\newcommand{\ZZ}{\mathbb{Z}}

\newcommand{\TT}{\mathbb{T}}

\newcommand{\CC}{\mathbb{C}}

\newcommand{\zz}{\bold{z}}
\newcommand{\ww}{\mathbf{w}}
\newcommand{\yy}{\mathbf{y}}

\DeclareMathOperator{\Der}{Der}
\DeclareMathOperator{\DR}{DR}
\DeclareMathOperator{\D}{D}

\DeclareMathOperator{\GL}{GL}
\DeclareMathOperator{\Mat}{Mat}

\DeclareMathOperator{\Eis}{Eis}

\DeclareMathOperator{\Hom}{Hom}

\DeclareMathOperator{\SL}{SL}

\DeclareMathOperator{\KS}{KS}

\DeclareMathOperator{\Res}{Res}

\DeclareMathOperator{\Id}{Id}

\newcommand{\tr}{\mathrm{tr}}

\newcommand{\dnorm}[1]{\lVert #1 \rVert}
\newcommand{\inorm}[1]{{\lvert #1 \rvert}}

\begin{document}

\title[Drinfeld modular forms]{On the partial derivatives of Drinfeld modular forms of arbitrary rank}

\author{Yen-Tsung Chen}
\address{Y-T. Chen \\ Department of Mathematics, National Cheng Kung University, Tainan City, Taiwan R.O.C.}
\email{ytchen.math@gmail.com}

\author{O\u{g}uz Gezm\.{i}\c{s}}
\address{O. Gezm\.{i}\c{s}\\Department of Mathematics, National Tsing Hua University, Hsinchu City 30042, Taiwan R.O.C.}
\email{gezmis@math.nthu.edu.tw}



\date{\today}



\thanks{The first author was partially supported by Professor Chieh-Yu Chang’s MOST Grant 107-2628-M-007-002-MY4 and the MOST Grant 110-2917-I-007-005. The second author acknowledges support by NSTC Grant 113-2115-M-007-001-MY3.}

\begin{abstract}
    In this paper, we obtain an analogue of the Serre derivation acting on the product of spaces of Drinfeld modular forms which generalizes the differential operator introduced by Gekeler in the rank two case. We further introduce a finitely generated algebra  $\mathcal{M}_r$ containing all the Drinfeld modular forms for the full modular group and show its stability under the partial derivatives. 
\end{abstract}
\maketitle
\section{Introduction}

Let $\mathfrak{M}_{k}(\Gamma)$ be the $\CC$-vector space of elliptic modular forms of weight $k$ for a congruence subgroup $\Gamma$ of $\SL_2(\ZZ)$ and $G_2$ be the false Eisenstein series of weight 2 normalized so that the first coefficient of its Fourier expansion is one. \textit{The Serre derivation} is the differential operator $D_k:\mathfrak{M}_k(\Gamma)\to \mathfrak{M}_{k+2}(\Gamma)$ given by 
\[
D_kf:=\frac{d}{dz}f-kG_2f.
\]
Naturally, one expects to generalize the Serre derivation to Hilbert modular forms or Siegel modular forms. However, although there are several differential operators defined by using the partial derivatives of Hilbert modular forms (see \cite{Coh77,Lee01,Ran56,Res72,Zag94}) which preserve modularity, there is no analogue of Serre derivation for this setting. On the other hand, to carry the notion of the Serre derivation for Siegel modular forms of degree $g$, Yang and Yin required the existence of a symmetric $g\times g$ matrix $G$ consisting of functions on the Siegel upper half plane and satisfying a certain functional equation \cite[Thm. 2.9]{YY15}. Nonetheless, Hofmann and Kohnen \cite{HK18} showed that such a matrix $G$ never exists and thus the generalization of $D_k$ also could  not be established for Siegel modular forms. For more details and explicit list of references, we refer the reader to aforementioned articles.

Unlike the case of classical setting briefly mentioned above, for Drinfeld modular forms, there is a generalization of the Serre derivation to more general objects and constructing it will be one of the main themes of the present paper. In particular, we introduce a multi-linear operator in Theorem \ref{T:I2} generalizing the Serre derivation defined by Gekeler \cite[Sec. 8]{Gek88} for Drinfeld modular forms of rank two to the setting of Drinfeld modular forms of arbitrary rank.

To state our results, we firstly introduce  some background. Let $\mathbb{F}_q$ be the finite field with $q$ elements where $q$ is a positive power of a prime $p$. Let $\theta$ be a variable over $\mathbb{F}_q$. We define $A$ to be the polynomial ring $\mathbb{F}_q[\theta]$ and $K$ to be its fraction field. We consider the non-archimedean norm $|\cdot|$ corresponding to the infinite place normalized so that $\inorm{\theta}=q$ and let $K_{\infty}$ be the completion of $K$ with respect to $|\cdot|$. We further let $\mathbb{C}_{\infty}$ be the completion of a fixed algebraic closure of $K_{\infty}$.

For $r\geq 2$, we define the \textit{Drinfeld period domain} $\Omega^r$ by \[
\Omega^r:=\mathbb{P}^{r-1}(\CC_{\infty})\setminus \{K_{\infty}\text{-rational hyperplanes}\}
\]
and identify each element $\zz\in \Omega^r$ by $\zz:=(z_1,\dots,z_r)^{\tr}\in \CC_{\infty}^r$ so that $z_r=1$. By definition, the entries $z_1,\dots,z_r$ are $K_{\infty}$-linearly independent. The $A$-module $A^r\zz\subset \CC_{\infty}$ generated over $A$ by the entries of $\zz$ forms an $A$-lattice of rank $r$ and hence, due to Drinfeld \cite{Dri74}, there exists a unique Drinfeld module, which we denote by  $\phi^{\zz}$, corresponding to $A^r\zz$ (see \S2 for more details). We note that $\Omega^r$ is a connected rigid analytic space whose admissible open subsets and open coverings can be explicitly described (see \cite{Dri74}, \cite[Sec. 4]{Hab21}, \cite{Put87} and \cite[Sec. 4]{FvdP04} for more details). Let $\{\Omega_n^r\}_{n=1}^{\infty}$ be such a covering of $\Omega^r$ (see \S3 for its explicit definition). Then we call $f:\Omega^r\to \CC_{\infty}$ \textit{a rigid analytic function} if, for each $n\in \ZZ_{\geq 1}$, its restriction to $\Omega_n^r$ is the uniform limit of rational functions which do not have any pole in $\Omega_n^r$. 

The study of Drinfeld modular forms for the rank two setting was initiated by Goss, in the 1980s, in his PhD thesis (see also \cite{Gos80})  and continued to be developed in following years by Gekeler (see \cite{Gek86,Gek88,Gek89,Gek90}). In \cite{BBP18}, based on the results of Pink in \cite{Pin13} and H\"{a}berli in his PhD thesis \cite{Hab21}, Basson, Breuer and Pink generalized the theory of Drinfeld modular forms to the arbitrary rank. They define Drinfeld modular forms, algebraically, to be global sections of particular ample invertible sheaves on the compactification of the Drinfeld moduli space and analytically to be rigid analytic functions on $\Omega^r$ satisfying a holomorphy condition at infinity and a certain automorphy condition with respect to an arithmetic subgroup of $\GL_{r}(K)$. They further studied the graded $\CC_{\infty}$-algebra of Drinfeld modular forms for $\GL_{r}(A)$ as well as the $\CC_{\infty}$-vector space of Drinfeld modular forms of any given weight and type (see \cite[Sec. 17]{BBP18}). We mention that the work of Pellarin \cite{Pel21} provides  another direction for the generalization of this theory via the study of vectorial Drinfeld modular forms and their deformations. For more details on the history of Drinfeld modular forms, the reader can refer to  \cite[Sec. 7]{BB17}, \cite{BBP18} and \cite[Sec. 1.1]{Per14}.

For $1\leq i \leq r$, we define \textit{the $i$-th coefficient form}  $g_i:\Omega^r\to \CC_{\infty}$, a Drinfeld modular form of weight $q^{i}-1$ and type 0 for $\GL_{r}(A)$, which sends each $\zz\in \Omega^r$ to the $i$-th coefficient of $\phi^{\zz}$. Note that when $i=r$, $g_r$ is indeed a nowhere-vanishing Drinfeld cusp form. Furthermore, Gekeler \cite{Gek17} introduced a Drinfeld cusp form $h_r:\Omega^r\to \CC_{\infty}$ of non-zero type which plays an essential role in the present work (see \S3 for more details). For any rigid analytic function $\mathfrak{h}:\Omega^r\to \CC_{\infty}$, let $\partial_{i}(\mathfrak{h})$ be the partial derivative of $\mathfrak{h}$ with respect to $z_i$ and for each $1\leq j \leq r-1$, we consider \textit{the false Eisenstein series $E_r^{[j]}:\Omega^r\to \CC_{\infty}$} given by
\[
E_r^{[j]}(\zz):=\frac{1}{g_r(\zz)}\partial_{j}(g_r(\zz)), \ \ \zz=(z_1,\dots,z_r)^{\tr}\in \Omega^r.
\] 
When $r=2$, we refer the reader to \cite[Sec. 8]{Gek88} for an explicit study of these functions. For $k\in \ZZ$, we further define the operator $\mathcal{D}_{j,k}$ by 
\[
\mathcal{D}_{j,k}(\mathfrak{h}):=\partial_{j}(\mathfrak{h})+k\mathfrak{h}E_{r}^{[j]}.
\]

Let $M_{k}^{m}(\Gamma)$ be the  $\CC_{\infty}$-vector space of Drinfeld modular forms  of weight $k\in \ZZ_{\geq 0}$ and type $m\in \ZZ/(q-1)\ZZ$ for a congruence subgroup $\Gamma$ of $\GL_{r}(A)$ (see \S3 for the explicit definition). Our first result, restated as Theorem \ref{T:RS} later, shows the existence of a differential operator acting on the product of the $\CC_{\infty}$-vector spaces of Drinfeld modular forms.

\begin{theorem}\label{T:I2}   For each $1\leq i \leq r-1$, let $\Gamma_i \leqslant \GL_r(A)$ be a congruence subgroup, $k_i\in \ZZ_{\geq 0}$ and $m_i\in \ZZ/(q-1)\ZZ$. Consider the operator $\D_{(k_1,\dots,k_{r-1})}$ on $M_{k_1}^{m_1}(\Gamma_1)\times \cdots \times M_{k_{r-1}}^{m_{r-1}}(\Gamma_{r-1})$ defined by 
	\[
	\D_{(k_1,\dots,k_{r-1})}(f_1,\dots,f_{r-1}):=\det \begin{pmatrix}
	\mathcal{D}_{1,k_1}(f_1)&\cdots &\mathcal{D}_{r-1,k_1}(f_1)\\
	\vdots & & \vdots  \\
	\mathcal{D}_{1,k_{r-1}}(f_{r-1})&\cdots &\mathcal{D}_{r-1,k_{r-1}}(f_{r-1})
	\end{pmatrix}.
	\]
	Then the following statements hold.
	\begin{itemize}
		\item[(i)] $\D_{(k_1,\dots,k_{r-1})}$ is a $\CC_{\infty}$-multi-linear derivation.
		\item[(ii)] $\D_{(k_1,\dots,k_{r-1})}(f_1,\dots,f_{r-1})\in M_{k_1+\dots+k_{r-1}+r}^{m_1+\dots+m_{r-1}+1}(\cap_{i=1}^{r-1}\Gamma_i)$.
	\end{itemize}
\end{theorem}

Motivated by Theorem \ref{T:I2}, we call  $\D_{(k_1,\dots,k_{r-1})}$ \textit{the Serre derivation}.  We remark that when $r=2$, Serre derivation has been extensively studied by Gekeler \cite{Gek88,Gek90} providing a striking analogy with the classical setting. On the other hand, the situation in the arbitrary rank setting, that is when $r\geq 3$, requires further innovative ideas, mainly for the following reason: In the rank two case, one can identify $\Omega^2$ with $\mathbb{C}_{\infty}\setminus K_{\infty}$ and hence realize rigid analytic functions to be single variable functions. Thus their partial derivatives are only based on one variable. On the other hand, when $r\geq 2$, our rigid analytic functions are viewed as $(r-1)$-many variable functions, and this leads to the consideration of partial derivatives with respect to $(r-1)$-many different variables. Nonetheless, taking the partial derivative of a Drinfeld modular form with respect to $(r-2)$-many of these variables will result a rigid analytic function which does not admit a Fourier expansion  around the infinity cusp. Hence one needs a delicate analysis to determine the behavior of $\D_{(k_1,\dots,k_{r-1})}(f_1,\dots,f_{r-1})$ around the cusps of a congruence subgroup. In order to do this, we first investigate the limiting behavior of partial derivatives of the uniformizer $u_{\Gamma(\mathfrak{n})}$ at infinity for any $\mathfrak{n}\in A\setminus \{0\}$ (see \S3 for its explicit definition) and our false Eisenstein series whenever $\zz$ is an element of $\Omega^{r}$ satisfying certain conditions (Lemma \ref{L:bounds2}). This, in Proposition \ref{P:Hol}, allows us to analyze the limiting behavior of the terms coming from the cofactor expansion of $\D_{(k_1,\dots,k_{r-1})}(f_1,\dots,f_{r-1})$ along the first column. Finally, combining our analysis with certain identities and functional equations obtained in Proposition \ref{P:1} and Lemma \ref{L:1}, we finish the proof of Theorem \ref{T:I2}.

In what follows, we present the second main result of the present paper. Let $\exp_{\phi^{\zz}}:\CC_{\infty}\to \CC_{\infty}$ be the exponential function of $\phi^{\zz}$. We consider \textit{the period matrix of $\phi^{\zz}$} given as 
\begin{equation}\label{E:period}
P_{\zz}:=
\begin{pmatrix}
-z_1& F^{\phi^{\zz}}_{\tau}(z_1) &\dots &F^{\phi^{\zz}}_{\tau^{r-1}}(z_1)\\
\vdots &\vdots & & \vdots \\
-z_r& F^{\phi^{\zz}}_{\tau}(z_r) &\dots &F^{\phi^{\zz}}_{\tau^{r-1}}(z_r)
\end{pmatrix}\in \GL_{r}(\CC_{\infty})
\end{equation}
where for each $1\leq j \leq r-1$, $F_{\tau^j}^{\phi^{\zz}}:\CC_{\infty}\to \CC_{\infty}$ is the unique entire function satisfying the functional equation
\begin{equation}\label{E:quasiperiodic}
F_{\tau^j}^{\phi^{\zz}}(\theta z)-\theta F^{\phi^{\zz}}_{\tau^j}(z)=\exp_{\phi^{\zz}}(z)^{q^j}, \  \ z\in \CC_{\infty}.
\end{equation}
Note that $F_{\tau},\dots,F_{\tau^{r-1}}$ are indeed \textit{quasi-periodic functions of $\phi^{\zz}$}, which play the analogous role of the quasi-periodic functions for elliptic curves (see \S5.1 for more details). For each $1 \leq i\leq r-1$ and $1\leq j \leq r$, consider 
\begin{align*}
L_{ij} &: \Omega^r\to \CC_{\infty}\\
&\zz \mapsto  P_{\zz}^{(i,j)}
\end{align*}
where $ P_{\zz}^{(i,j)}$ is the $(i,j)$-cofactor of the period matrix $P_{\zz}$. We finally let $(-\theta)^{1/(q-1)}$  be a fixed $(q-1)$-st root of $-\theta$ and define \textit{the Carlitz period}, analogous to $2\pi i $ in the classical setting (see \cite[Sec. 2.5]{Tha04}), by 
\[\tilde{\pi}:=\theta(-\theta)^{1/(q-1)}\prod_{i=1}^{\infty}\Big(1-\theta^{1-q^i}\Big)^{-1}\in \CC_{\infty}^{\times}.
\]

Our next theorem, restated as Theorem \ref{T:GM2} later, can be described as follows.
\begin{theorem}\label{T:I1} For each $1\leq i,j \leq r-1$ and $\zz\in \Omega^r$, we have
		 \[
         \partial_{j}(g_i)(\zz)=E_r^{[j]}(\zz)g_i(\zz)+\tilde{\pi}^{q+\dots+q^{r-1}}h_r(\zz)L_{j(i+1)}(\zz).
         \]
\end{theorem}

 When $r=2$, Theorem \ref{T:I1} was established by Gekeler in \cite[Sec. 9]{Gek88} and \cite[(8.4)]{Gek89} by using the ``$k/12$ formula for Drinfeld modular forms'' and an analysis of the behavior at a neighborhood of infinity of the function $h_2$. Our method, however, follows a completely different path, which also provides a new proof in the rank two case. More precisely, we study the work of Pellarin in \cite[Sec. 8]{Pel19} on the deformation of vector-valued modular forms and combine it with the theory of Drinfeld modules over the Tate algebra developed in \cite{AnglesPellarinTavares16} and \cite{GP19}. Later on, we use the explicit relations between the Eisenstein series and the coefficients of the exponential and the logarithm series of Drinfeld modules to complete the proof (see \S5.2 for more details). 

An immediate corollary of Theorem \ref{T:GM2}, later restated as Proposition \ref{P:hfunc}, which follows from an explicit algebraic relation between  $h_r$ and $L_{ij}$ (see \eqref{E:DET}) yields a different characterization of the $h$-function of Gekeler from that given in \cite[Sec. 3.7]{Gek17}.

\begin{corollary}\label{C:RS} We have 
	\[
	\D_{(q-1,\dots,q^{r-1}-1)}(g_1,\dots,g_{r-1})(\zz)=\tilde{\pi}^{q+\dots+q^{r-1}}h_r(\zz).
	\]
\end{corollary}

Our next goal is to introduce a particular $\CC_{\infty}$-algebra invariant under the partial derivatives. We define  $\mathcal{M}_r$ by
\[
\mathcal{M}_r:=\mathbb{C}_\infty[g_1,\dots,g_{r-1},h_r, \partial_j(g_i), E_r^{[j]}\mid 1\leq i,j\leq r-1].
\]
One of the main purposes of the present paper is to initiate a study towards understanding the differential structure of $\mathcal{M}_r$. We emphasize that $\mathcal{M}_2$ is the $\CC_{\infty}$-algebra of \textit{Drinfeld quasi-modular forms} in the rank two setting, defined by Bosser and Pellarin, containing the graded $\CC_{\infty}$-algebra of Drinfeld modular forms for $\GL_2(A)$ and is stable under \textit{the hyperderivatives of degree $n\geq 0$}, a certain higher derivatives which replace the usual derivatives in the positive characteristic case  (see \S5.3 for more details).

Our result, later stated as Theorem \ref{Thm:QM} and Theorem \ref{T:eq}, can be given as follows.
\begin{theorem}\label{T:I3} The $\CC_{\infty}$-algebra $\mathcal{M}_r$ is stable under the partial derivatives and strictly contains the graded $\CC_{\infty}$-algebra of Drinfeld modular forms for $\GL_{r}(A)$. Moreover, if we let $\widetilde{\mathcal{M}_r}$ be the $\CC_{\infty}$-algebra  given by 
\[
\widetilde{\mathcal{M}_r}:=\mathbb{C}_\infty[g_i,h_rL_{ij}\mid 1\leq i\leq r-1,~1\leq j\leq r],
\]
then $\widetilde{\mathcal{M}_r}=\mathcal{M}_r$.
\end{theorem}

\begin{remark}  Although it is still not clear how to define quasi-modular forms explicitly in the higher rank setting, due to Theorem \ref{T:I3}, we predict that $\mathcal{M}_r$ is the ring of such forms. Moreover, to determine the transcendence degree of $\mathcal{M}_r$ as a $\CC_{\infty}$-algebra, it would be interesting to investigate a suitable adaptation of the strategy used in the work of Bertrand and Zudilin on Siegel modular forms \cite[Thm. 1]{BZ03}. Inspired by the rank two setting, where the transcendence degree of $\mathcal{M}_r$ is $3$ \cite[Thm. 1]{BP08}, and the equality $\widetilde{\mathcal{M}_r}=\mathcal{M}_r$, it is natural to expect that the transcendence degree of $\mathcal{M}_r$ is $r^2-1$. However, this requires a more technical approach which would be beyond the scope of the present paper and hence we leave this problem as one of the topics for the  future research.
\end{remark}


The outline of the paper is as follows. In \S2, we introduce  basic properties of Drinfeld modules. In \S3, we explain some details on Drinfeld modular forms and provide several examples including the $h$-function of Gekeler.  In \S4, we provide a proof for Theorem \ref{T:I2}. Finally, in \S5, we  obtain Theorem \ref{T:I1} and Theorem \ref{T:I3} after analyzing the deformation of Eisenstein series.

\subsection*{Acknowledgments} The authors would like to express their sincere gratitude to Jing Yu for his continuous interest on the present work as well as for many fruitful discussions on the work done in \cite{CG21}, which motivate the results obtained in the paper. The authors are indebted to Mihran Papikian for his suggestions which lead to improve the results in the paper. The authors thank to Gebhard B\"{o}ckle, Chieh-Yu Chang, Ernst-Ulrich Gekeler, Peter Gr\"{a}f, Simon H\"{a}berli, Andreas Maurischat and Federico Pellarin for many  valuable comments on the results of the paper. The authors are also grateful to the referee for several valuable suggestions on the structure and the organization of the manuscript, especially for the alternative arguments which shorten the proof of Theorem \ref{T:I2} substantially. The first author also thanks Texas A\&M University for its hospitality. 

\section{Preliminaries and Background}
In this section, we introduce the necessary background for Drinfeld modules. For more details, we refer the reader to \cite[\S2]{Gek88}, \cite[\S3, \S4]{Goss}, \cite{Gek90} and \cite[\S 2]{Tha04}.

 For any $\mathbb{F}_q$-algebra $L\subseteq \CC_{\infty}$ containing $K$, we define the non-commutative ring $L[[\tau]]$ of power series in $\tau$ subject to the condition 
\[
\tau z =z^q \tau,\ \ z\in L.
\]
Moreover, we let $L[\tau]\subset L[[\tau]]$ be the ring of polynomials in $\tau$ with coefficients in $L$. It has an action on $L$ given by 
\[
\sum c_i\tau^i \cdot z:=\sum c_iz^{q^i}\in L.
\]

Let $r\in \ZZ_{\geq 1}$. \textit{A Drinfeld module $\phi$ of rank $r$ over $L$} is an $\mathbb{F}_q$-algebra homomorphism $\phi:A\to L[\tau]$ given by
\[
\phi_\theta:=\theta+g_{1}\tau+\dots +g_{r}\tau^{r}
\]
such that $g_{r}\neq 0$. For each $1\leq i \leq r$, we call $g_{i}$ \textit{the $i$-th coefficient of $\phi$}. 

For each $\phi$, there exists a unique series $\exp_{\phi}=\sum_{i\geq 0}\alpha_{i}\tau^i\in \CC_{\infty}[[\tau]]$, called \textit{the exponential series of $\phi$}, satisfying $\alpha_0=1$ and 
\[
\exp_{\phi} \theta=\phi_\theta \exp_{\phi}.
\] 
Moreover it induces an $\mathbb{F}_q$-linear entire function $\exp_{\phi}:\CC_{\infty}\to \CC_{\infty}$
given by $\exp_{\phi}(z)=\sum_{i\geq 0}\alpha_{i}z^{q^i}$ for each $z\in \CC_{\infty}$. Similarly, one can have \textit{the logarithm series $\log_{\phi}$ of $\phi$} which is the formal inverse of $\exp_{\phi}$ and given by the infinite series $\log_{\phi}=\sum_{i\geq 0}\beta_{i}\tau^i\in \CC_{\infty}[[\tau]]$ satisfying $\beta_0=1$ and 
\[
\theta\log_{\phi} =\log_{\phi}\phi_{\theta}.
\] 

\textit{An $A$-lattice $\Lambda$ of rank $r$} is a free $A$-module of rank $r$ which is discrete, that is its intersection with any ball in $\CC_{\infty}$ of finite radius is finite. For such $\Lambda$, consider the function $e_{\Lambda}:\CC_{\infty}\to \CC_{\infty}$ defined by 
\[
e_{\Lambda}(z):=z\prod_{\substack{\lambda\in \Lambda\\\lambda\neq 0}}\Big(1-\frac{z}{\lambda}\Big), \ \ z\in \CC_{\infty}.
\]
For each $a\in A$, one can define a map $\phi_a^{\Lambda}:\CC_{\infty}\to \CC_{\infty}$ so that the following diagram commutes:
\begin{displaymath}
\begin{tikzcd}[column sep=large]
0 \arrow{r} &\Lambda \arrow{r} \arrow{d}{\cdot a}
&\CC_{\infty}\arrow{r}{e_{\Lambda}}\arrow{d}{\cdot a}
&\CC_{\infty}\arrow{d}{\phi_a^{\Lambda}}\arrow{r}&0\\
0 \arrow{r} &\Lambda \arrow{r}&\CC_{\infty}\arrow{r}{e_{\Lambda}}&\CC_{\infty} \arrow{r}&0.
\end{tikzcd}
\end{displaymath}
Indeed, for any Drinfeld module $\phi$ of rank $r$ over $\CC_{\infty}$, there exists a unique $A$-lattice of rank $r$, called \textit{the period lattice of $\phi$}, so that the above diagram is commutative. Moreover, due to Drinfeld \cite{Dri74}, we know that the association $\Lambda\to \phi^{\Lambda}$ defines a bijection between the set of $A$-lattices of rank $r$  in $\CC_{\infty}$ and the set of Drinfeld modules of rank $r$ over $\CC_{\infty}$. Furthermore, we call each non-zero element of $\Lambda$ \textit{a period} of $\phi^{\Lambda}$.

In what follows, we introduce an $A$-lattice invariant which will be fundamental for the present work. Let $\Lambda$ be an $A$-lattice of rank $r$. We let $\Eis_0(\Lambda):=-1$ and define \textit{the Eisenstein series of weight $k\in \ZZ_{\geq 1}$ for $\Lambda$} to be the infinite sum given by 
\begin{equation}\label{Eq:Eisenstein_Series}
\Eis_k(\Lambda):=\sum_{\substack{\lambda\in \Lambda\\\lambda\neq 0}}\frac{1}{\lambda^k}\in \CC_{\infty}.
\end{equation}
 Note that when $k$ is not divisible by $q-1$, we have $\Eis_k(\Lambda)=0$. Moreover, for a variable $X$ over $\mathbb{C}_{\infty}$, by \cite[(2.8)]{Gek88}, we get
\begin{equation}\label{E:eis1}
\frac{X}{\exp_{\phi^{\Lambda}}(X)}=1-\sum_{k=1}^{\infty}\Eis_k(\Lambda)X^k.
\end{equation}
Furthermore, as described in \cite[(2.9)]{Gek88}, we have the following relation between the Eisenstein series and the coefficients of the logarithm series $\log_{\phi^{\Lambda}}=\sum_{i\geq 0}\beta_i\tau^i$:
\begin{equation}\label{E:eis2}
\Eis_{q^k-q^j}(\Lambda)=-\beta_{k-j}^{q^j}, \ \ j,k\geq 0.
\end{equation}



\section{Drinfeld modular forms}
In this section, we review the notion of Drinfeld modular forms of arbitrary rank. We refer the reader to \cite{BBP18} for related details. 

For any integer $r\geq 2$, let $\mathbb{P}^{r-1}(\mathbb{C}_\infty)$ be the projective space of dimension $r-1$ with coefficients in $\CC_{\infty}$. We define the Drinfeld period domain
\[
\Omega^r=\mathbb{P}^{r-1}(\CC_{\infty})\setminus \{K_{\infty}\text{-rational hyperplanes}\}
\]
 and identify any of its elements as $\zz=(z_1,\dots,z_r)^{\tr}\in \CC_{\infty}^{r}$ whose entries are $K_{\infty}$-linearly independent and normalized so that $z_r=1$. 
 
For later use, we now briefly explain the rigid analytic structure of $\Omega^r$.  Let $H$ be a $K_\infty$-rational hyperplane in $\mathbb{P}^{r-1}(\mathbb{C}_\infty)$. We choose a \emph{unimodular} linear form
\[
\ell_H(X_1,\dots,X_r):=a_1X_1+\cdots+a_rX_r\in K_\infty[X_1,\dots,X_r]
\]
to be the defining equation of $H$, that is, $\max_{1\leq i\leq r}\{\inorm{a_i}\}=1$. Then setting $\ell_H(\zz):=a_1z_1+\dots+a_rz_r$, we see that $\inorm{\ell_{H}(\zz)}$ is well-defined for any $\zz\in\Omega^r$ and is independent of the choice of $\ell_H$. Let $|\mathfrak{y}|_{\infty}:=\max_{i=1}^r\inorm{y_i}$ where $\mathfrak{y}=(y_1,\dots,y_r)^{\tr}\in \mathbb{C}_{\infty}^r$. For each $n\geq 1$, we define
\[
\Omega_n^r:=  \{\zz\in\Omega^r\mid\inorm{\ell_H(\zz)}\geq q^{-n}|\zz|_{\infty} \ \ \text{for any $K_{\infty}$-rational hyperplane $H$}\}.
\]
In fact, $\{\Omega_n^r\}_{n=1}^{\infty}$ forms an admissible covering of $\Omega^r$ \cite[Prop.~3.4]{BBP18} (cf. \cite[Prop.~1]{SS91}). 


For any $\gamma=(a_{ij})\in \GL_r(K_\infty)$, define
\begin{equation}\label{E:action}
\gamma \cdot \zz:=\Big(\frac{a_{11}z_1+\dots+ a_{1r}z_r}{a_{r1}z_1+\dots+ a_{rr}z_r},\dots,\frac{a_{(r-1)1}z_1+\dots +a_{(r-1)r}z_r}{a_{r1}z_1+\dots+ a_{rr}z_r},1\Big)^{\tr}\in \Omega^r.
\end{equation}
We further set
\[
j(\gamma,\zz):=a_{r1}z_1+\dots+ a_{rr}z_r\in \CC_{\infty}^{\times}.
\]

For any $\mathfrak{n}\in A\setminus\{0\}$, we define \textit{a principal congruence subgroup} $\Gamma(\mathfrak{n})$ to be the kernel of the projection $\GL_{r}(A)\mapsto \GL_{r}(A/\mathfrak{n}A)$. Note that $\Gamma(1)=\GL_r(A)$. Moreover, we call $\Gamma\subseteq \GL_{r}(A)$ \textit{a congruence subgroup} if it contains $\Gamma(\mathfrak{n})$ for some non-zero $\mathfrak{n}$. A rigid analytic function $f:\Omega^r\to \CC_{\infty}$ is called \textit{a weak modular form  of weight $k\in \ZZ$ and type $m\in \ZZ/(q-1)\ZZ$ for $\Gamma$}  if it satisfies
\[
f(\gamma\cdot \zz)=j(\gamma,\zz)^k\det(\gamma)^{-m}f(\zz),\ \  \gamma\in \Gamma, \ \ \zz \in \Omega^r.
\]

Consider the map 
$
\iota:K^{r-1}\to \GL_r(K)$ given by
\[
\iota:(a_2,\dots,a_r)\mapsto
\left(\begin{smallmatrix}
1 & a_2 & \cdots &a_r\\
& 1 & &\\
& & \ddots &\\
& & & 1
\end{smallmatrix}\right).
\]
For a congruence subgroup $\Gamma$, we define $\Gamma_{\iota}:=\Gamma\cap\iota(K^{r-1})$ and say that  $f:\Omega^r\to\mathbb{C}_\infty$ is \textit{$\Gamma_{\iota}$-invariant} if  $f(\gamma\cdot \zz)=f(\zz)$ for all $\gamma\in \Gamma_{\iota}$. Note that any weak modular form $f$ of weight $k$ and type $m$ for $\Gamma$ is automatically $\Gamma_{\iota}$-invariant. For $\zz=(z_1,\dots,z_r)^{\tr}\in\Omega^r$, we set $\ww=(z_2,\dots,z_r)^{\tr}$. Note that  $\iota^{-1}(\Gamma_{\iota})\cap A^{r-1}$ is of finite index in $A^{r-1}$ and $\iota^{-1}(\Gamma_{\iota})$. Hence 
\[
\tilde{\Gamma}\ww:=\{z_2a_2+\cdots+z_{r}a_{r}\mid (a_2,\dots,a_{r})\in\iota^{-1}(\Gamma_{\iota})\}\subset\mathbb{C}_\infty
\]
forms an $A$-lattice of rank $r-1$ in $\mathbb{C}_{\infty}$. For any $A$-lattice $\Lambda$ and $c\in \mathbb{C}_{\infty}^{\times}$, we let $c\Lambda:=\{c\lambda \ \ | \ \ \lambda\in \Lambda\}$. Now considering $\tilde{\pi}\tilde{\Gamma}\ww\subset \CC_{\infty}$, we also define 
\[
u_\Gamma(\zz):=e_{\tilde{\pi}\tilde{\Gamma}\ww}(\tilde{\pi}z_1)^{-1}\in \CC_{\infty}^{\times}.
\]
When $\Gamma$ is a principal congruence subgroup, we can describe $u_{\Gamma}$ more precisely as follows: Let $\mathfrak{n}\in A\setminus\{0\}$ and $A^{r-1}\ww$ be the $A$-lattice of rank $r-1$ generated by the entries of $\ww$ over $A$. Then we have
\[
u_{\Gamma(\mathfrak{n})}(\zz)=\exp_{\tilde{\pi}\mathfrak{n}A^{r-1}\ww}(\tilde{\pi}z_1)^{-1}.
\]
Since $\exp_{c\Lambda}(cz)=c\exp_{\Lambda}(z)$ for any $c\in \mathbb{C}_{\infty}^{\times}$ and $z\in \mathbb{C}_{\infty}$ (see \cite[Sec. 4]{Gek89}), we have
\begin{equation}\label{E:cong}
u_{\Gamma(\mathfrak{n})}(\zz)=\exp_{\tilde{\pi}\mathfrak{n}A^{r-1}\ww}(\tilde{\pi}z_1)^{-1}=\mathfrak{n}^{-1}\tilde{\pi}^{-1}\exp_{\phi^{\ww}}(\mathfrak{n}^{-1}z_1)^{-1}
\end{equation}
where we set $\phi^{\ww}$ to be the Drinfeld module corresponding to the $A$-lattice $A^{r-1}\ww$. To ease the notation, we further let $u(\zz):=u_{\Gamma(1)}(\zz)$.


By \cite[Prop. 5.4]{BBP18} we know that for any $\Gamma_{\iota}$-invariant rigid analytic function $f:\Omega^r\to \CC_{\infty}$, for each $n\in \ZZ$, there exists a unique rigid analytic function $f_n:\Omega^{r-1}\to \CC_{\infty}$  such that the series
\begin{equation}\label{E:expansion}
\sum_{n\in \ZZ}f_n(\ww)u_\Gamma(\zz)^n
\end{equation}
converges to $f(\zz)$ on some neighborhood of infinity and its admissible subsets (see \cite[Def. 4.12]{BBP18} for the explicit definition of a neighborhood of infinity). We call the infinite sum given in \eqref{E:expansion} \textit{the $u_{\Gamma}$-expansion of $f$}. Note that when $r=2$,  $f_n$ is a constant in $\CC_{\infty}$ for each $n$. 

For any $\nu \in \GL_{r}(A)$, we know that if $f$ is a weak modular form of weight $k$ and type $m$ for $\Gamma$, then the rigid analytic function $f_{|_{k,m}}[\nu]:\Omega^r\to \CC_{\infty}$ given by $f_{|_{k,m}}[\nu](\zz):=j(\nu,\zz)^{-k}\det(\nu)^mf(\nu\cdot \zz)$ is a weak modular form of weight $k$ and type $m$ for the congruence subgroup $\nu^{-1}\Gamma \nu$. Moreover, we call $f$ \textit{a Drinfeld modular form  of weight $k$ and type $m$ for $\Gamma$} if the function $f_n$ in the $u_{\nu^{-1}\Gamma \nu}$-expansion of $f_{|_{k,m}}[\nu]$ for all $\nu \in \GL_{r}(A)$ is identically zero when $n<0$, in other words, $f$ is \textit{holomorphic at infinity with respect to $(\nu^{-1}\Gamma \nu)_{\iota}$}. Furthermore, we call $f$ \textit{a Drinfeld cusp form  of weight $k$ and type $m$ for $\Gamma$} if, in addition, $f_0\equiv 0$ in the $u_{\nu^{-1}\Gamma \nu}$-expansion of $f_{|_{k,m}}[\nu]$ for each $\nu \in \GL_{r}(A)$. We define $M^m_k(\Gamma)$ to be the $\mathbb{C}_\infty$-vector space spanned by all Drinfeld modular forms  of weight $k$ and type $m$ for $\Gamma$.

In what follows, we give some examples of Drinfeld modular forms.
\begin{example}\label{Ex:1}
	\begin{itemize}
		\item[(i)] Let $\zz=(z_1,\dots,z_r)^{\tr}\in \Omega^r$ and recall that $A^r\zz$ is the $A$-lattice of rank $r$ generated by the entries of $\zz$ over $A$. Let $k$ be a positive integer divisible by $q-1$ and $\zz\in \Omega^r$. It is shown in \cite[Prop. 13.3]{BBP18} that the map $\Eis_k:\Omega^r\to \CC_{\infty}$ sending $\zz\mapsto \Eis_k(A^r\zz)$ is a Drinfeld modular form  of weight $k$ and type 0 for $\GL_r(A)$.
		\item[(ii)] Let $\phi^{\zz} $ be the Drinfeld module corresponding to $A^r\zz$. Set 
		\[
		\phi^{\zz}_\theta:=\theta+g_1(\zz)\tau+\dots+g_{r}(\zz)\tau^{r}.
		\]
		Then, for any $1\leq i \leq r$, the function $g_{i}:\Omega^r\to \CC_{\infty}$, which we call \textit{the $i$-th coefficient form}, mapping $\zz\mapsto g_i(\zz)$ is a Drinfeld modular form of weight $q^i-1$ and type 0 for $\GL_r(A)$ \cite[Prop. 15.12]{BBP18}. In addition, we remark that  $g_r$ is indeed a nowhere-vanishing Drinfeld cusp form on $\Omega^r$. Moreover, setting $g_0=\theta$ and $\Eis_{0}=-1$, by \cite[(2.10)]{Gek88}, we have 
		\begin{equation}\label{E:EisCoef}
		\theta\Eis_{q^i-1}(\zz)=\sum_{k=0}^{i}\Eis_{q^k-1}(\zz)g_{i-k}^{q^k}(\zz).
		\end{equation}

		\item[(iii)] In what follows, we review the definition of one fundamental example of Drinfeld cusp forms which was introduced  by Gekeler \cite{Gek17}. Let $\mu=(\mu_1,\dots,\mu_r) \in \mathbb{F}_q^{r}\setminus \{(0,\dots,0)\}$. We call $\mu$ \textit{monic} if $\mu_k=1$ for the largest subscript $k$ satisfying $\mu_k\neq 0$. For any  $\zz=(z_1,\dots,z_r)^{\tr}\in \Omega^r$, let
		\[
		\Eis_{\mu}(\zz):=\sum_{\substack{(a_1,\dots,a_r)\in K^{r}\\ (a_1,\dots,a_r)\equiv \theta^{-1}\mu \pmod{A^r} }}\frac{1}{a_1z_1+\dots+a_rz_r}.
		\]
		 Then \textit{the $h$-function $h_r:\Omega^r\to \CC_{\infty}$ of Gekeler} is defined by
		\[
	h_r(\zz):=\tilde{\pi}^{\frac{1-q^r}{q-1}}(-\theta)^{1/(q-1)}\prod_{\substack{\mu\in \mathbb{F}_q^{r}\\\mu \text{ monic}}}\Eis_{\mu}(\zz).
	\]
We note that $h_r$ is a nowhere-vanishing Drinfeld cusp form  of weight $q^r-1/(q-1)$ and type 1 for $\GL_r(A)$. In other words, for each $\gamma\in\GL_r(A)$, we have
		\begin{equation}\label{E:FEh}
		h_r(\gamma\cdot \zz)=\det(\gamma)^{-1}j(\gamma,\zz)^{\frac{q^r-1}{q-1}}h_r(\zz)	.
		\end{equation}
		Furthermore, by \cite[Thm. 3.8]{Gek17}, we obtain
		\begin{equation}\label{E:hfunc}
		g_r(\zz)=\tilde{\pi}^{q^r-1}(-1)^{r-1}h_r(\zz)^{q-1}.
		\end{equation}
	\end{itemize}
\end{example}

\section{The differential operator $\D_{(k_1,\dots,k_{r-1})}$} 

In this section, we show the existence of a differential operator which also generalizes the notion of \textit{the Serre derivation} due to Gekeler to the higher rank setting. 
Our main goal is to prove Theorem \ref{T:RS}. For this, in what follows, we will first establish several results which are needed to determine the regularity at  infinity of $\D_{(k_1,\dots,k_{r-1})}$ acting on the tuple of Drinfeld modular forms.

Let $\mathcal{Y}$ be an affinoid open subspace of $\Omega^r$ and $f:\mathcal{Y}\to\mathbb{C}_\infty$ be a rigid analytic function. We define
    \[
        \dnorm{f}_{\mathcal{Y}}:=\sup_{\mathbf{y}\in\mathcal{Y}}\{|f(\mathbf{y})|\}.
    \]
    Note that by \cite[\S3.8 Cor. 2]{BGR84}, we have $\dnorm{f}_{\mathcal{Y}}<\infty $ for any affinoid open subspace $\mathcal{Y}$ and any rigid analytic function $f$ on $\mathcal{Y}$ \cite[\S5]{BBP18}.
Recall the admissible covering $\{\Omega_n^r\}_{n=1}^{\infty}$ of $\Omega^r$ given in \S3. For each $n\geq 1$, we set
$
\dnorm{f}_n:=\dnorm{f}_{\Omega_n^r}.
$ For any $\zz\in\Omega_n^r$ and $\eta>0$, we consider the closed polydisk $D(\zz,\eta)$ centered at $\zz$ and of radius $\eta |\zz|_{\infty}$ given by 
    \[
        D(\zz,\eta):=\{\yy\in\Omega^{r}\mid|\zz-\yy|_{\infty}\leq\eta|\zz|_{\infty}\}.
    \]
    Our next lemma provides an estimation of the size of the partial derivatives for rigid analytic functions on $\Omega_n^r$.
    \begin{lemma}\label{L:Bnd1}
        Let $f:\Omega_n^r\to\mathbb{C}_\infty$ be a rigid analytic function. Then for each $\zz\in\Omega_n^r$ and $\eta\in\mathbb{Q}$ with $0<\eta<q^{-n}$, we have
        \[
            D(\zz,\eta)\subset\Omega_n^r.
        \]
        Furthermore, for any $1\leq j\leq r-1$, we have
        \[
            \dnorm{\partial_jf}_{D(\zz,\eta)}\leq\eta^{-1}\dnorm{f}_n .
        \]
    \end{lemma}
    \begin{proof}        Let $\zz=(z_1,\dots,z_r)^{\tr}\in\Omega_n^r$ and $\yy=(y_1,\dots,y_r)^{\tr}\in D(\zz,\eta)$. Then
        for any $K_\infty$-rational hyperplane $H$ with $\ell_H(X_1,\dots,X_r)=a_1X_1+\cdots+a_rX_r$, we have
        \begin{align*}
        \inorm{\ell_H(\yy)}&=\inorm{a_1y_1+\cdots+a_ry_r}\\
            &=\inorm{a_1(y_1-z_1)+\cdots+a_r(y_r-z_r)+a_1z_1+\cdots +a_rz_r}\\
            &=|\ell_H(\zz)|\\
            &\geq q^{-n}|\zz|_{\infty}
        \end{align*}
        where the third and forth equality follow from the facts that
        \[
            \inorm{a_1(y_1-z_1)+\cdots+a_r(y_r-z_r)}\leq\max_{1\leq i\leq r}\{|a_i||y_i-z_i|\}\leq\eta|\zz|_\infty<q^{-n}|\zz|_{\infty}
        \]
        and
        \[
            \inorm{a_1z_1+\cdots+a_rz_r}=\inorm{\ell_H(\zz)}\geq q^{-n}|\zz|_{\infty}.
        \]
On the other hand, we claim that $|\yy|_{\infty}\leq |\zz|_{\infty}$. To prove our claim, assume to the contrary that $|\yy|_{\infty}> |\zz|_{\infty}$. Then there exists $1\leq i \leq r$ such that $|\yy|_{\infty}=|y_i|>|z_j|$ for any $1\leq j \leq r$. However this implies that 
\[
|\zz-\yy|_{\infty}=|z_i-y_i|=|y_i|=|\yy|_{\infty}>|\zz|_{\infty}
\]
which would contradict to the fact that $\yy \in D(\zz,\eta)$ for $0<\eta<q^{-n}$. Thus, we have 
\[
\inorm{\ell_H(\yy)}\geq q^{-n}|\zz|_{\infty}\geq q^{-n}|\yy|_{\infty}.
\]
Hence, $\yy\in \Omega_n^r$, which also implies that $D(\zz,\eta)\subset\Omega_n^r$. For the second assertion, note that, since $f$ is a rigid analytic function on $\Omega_n^r$, its restriction to  $D(\zz,\eta)$ can be expressed as
        \[
            f|_{D(\zz,\eta)}=\sum_{I=(i_1,\dots,i_r)\in\mathbb{Z}_{\geq 0}^r}C_I(X_1-z_1)^{i_1}\cdots(X_r-z_r)^{i_r},
        \]
        where $C_I\in\mathbb{C}_\infty$ and $C_I\left(\eta|\zz|_\infty\right)^{|I|}\to 0$ as $|I|:=i_1+\cdots+i_r\to\infty$. In particular, since $D(\zz,\eta)\subset\Omega_n^r$, by the maximum modulus principle \cite[\S2.2,~Prop.~5]{Bos14},  we have $|C_I|\leq \eta^{-|I|}|\zz|_\infty^{-|I|}\|f\|_{D(\zz,\eta)}\leq \eta^{-|I|}|\zz|_\infty^{-|I|}\|f\|_{n}$. Thus,
        \begin{align*}
            \dnorm{\partial_jf}_{D(\zz,\eta)}&=\sup_{\yy\in D(\zz,\eta)}\{\big|\sum_{I=(i_1,\dots,i_r)\in\mathbb{Z}_{\geq 0}^r}i_jC_I(y_1-z_1)^{i_1}\cdots(y_j-z_j)^{i_j-1}\cdots(y_r-z_r)^{i_r}\big|\}\\
            &\leq\sup_{I=(i_1,\dots,i_r)\in\mathbb{Z}_{\geq 0}^r,~i_j\geq 1}\{|C_I|\eta^{|I|-1}|\zz|_\infty^{|I|-1}\}\\
            &\leq\eta^{-1}|\zz|_\infty^{-1}\dnorm{f}_n.
        \end{align*}
        Finally, $|\zz|_{\infty}\geq 1$ yields the desired inequality.
    \end{proof}

\subsection{Certain identities and functional equations} Our goal in this subsection is to prove Proposition \ref{P:1} which has a crucial importance to prove Theorem \ref{T:I2}. We would like to thank the anonymous referee for a simpler argument of its proof which we provide with details in what follows. We refer the reader to the previous version of the manuscript \cite[Prop.~4.23]{CG22} for an alternative idea for the proof of Proposition \ref{P:1}. 

Throughout this subsection, we let $r\geq 2$,  $\zz=(z_1,z_2,\dots,,z_r)^{\tr}\in \Omega^r$ and hence $\ww=(z_2,\dots,z_r)^{\tr}\in \Omega^{r-1}$. For $\gamma=(a_{ij})_{i,j}\in \GL_r(\mathbb{C}_{\infty})$, $1\leq i \leq r$ and $\zz\in \Omega^r$, we set
\[
\delta_i(\gamma,\zz):=a_{i1}z_1+\cdots+a_{ir}z_r.
\]
Observe that when $\gamma\in \GL_r(A)$ and $\zz\in \Omega^r$, we have $\delta_r(\gamma,\zz)=j(\gamma,\zz)$. For $1\leq j,l\leq r-1$, we also set 
\[
\mathfrak{c}^{\gamma}_{jl}(\zz):=c^{\gamma}_{jl}-c_{jr}^{\gamma}z_l
\]
where $c_{jl}^{\gamma}$ (resp. $c_{jr}^{\gamma}$) is the $(j,l)$-cofactor (resp. $(j,r)$-cofactor) of $\gamma$.

\begin{lemma}\label{L:det} We have 
\[
\det((a_{ij}\delta_r(\gamma,\zz)-a_{rj}\delta_i(\gamma,\zz))_{1\leq i,j\le r-1})=\delta_r(\gamma,\zz)^{r-2}\det(\gamma)z_r=j(\gamma,\zz)^{r-2}\det(\gamma).
\]
\end{lemma}
\begin{proof}
 For any $1\leq \ell\leq r$, set 
$
a_\ell:=(a_{1\ell},\dots,a_{(r-1)\ell})^{\tr}$
and $
\delta(\gamma,\zz)=(\delta_1(\gamma,\zz),\dots,\delta_{r-1}(\gamma,\zz))^{\tr}
$
so that 
\[
\delta(\gamma,\zz)=z_1a_1+\cdots+z_ra_r.
\]
Observe that 
\begin{multline}\label{E:Eq18}
\det((a_{ij}\delta_r(\gamma,\zz)-a_{rj}\delta_i(\gamma,\zz))_{i,j})\\
=|a_1\delta_r(\gamma,\zz)-a_{r1}\delta(\gamma,\zz) \ \ a_2\delta_r(\gamma,\zz)-a_{r2}\delta(\gamma,\zz) \ \ \cdots \ \ a_{r-1}\delta_r(\gamma,\zz)-a_{r(r-1)}\delta(\gamma,\zz)|.
\end{multline}
Here, by the notation in the right hand side of \eqref{E:Eq18}, we mean the determinant of the matrix whose $j$-th column is given by $a_{j}\delta_r(\gamma,\zz)-a_{rj}\delta(\gamma,\zz)\in \Mat_{(r-1)\times 1}(\mathbb{C}_{\infty})$. We now expand the right hand side of \eqref{E:Eq18} by using the multi-linearity of determinants and we choose either $a_j\delta_r(\gamma,\zz)$ or $-a_{rj}\delta(\gamma,\zz)$. It is crucial to note that choosing $-a_{rj}\delta(\gamma,\zz)$ twice yields the corresponding term to vanish. Thus, we obtain
\begin{multline}\label{E:Eq19}
    |a_1\delta_r(\gamma,\zz)-a_{r1}\delta(\gamma,\zz) \ \ a_2\delta_r(\gamma,\zz)-a_{r2}\delta(\gamma,\zz) \ \ \cdots \ \ a_{r-1}\delta_r(\gamma,\zz)-a_{r(r-1)}\delta(\gamma,\zz)|\\
    =\delta_{r}(\gamma,\zz)^{r-1}|a_1 \ \ \cdots \ \ a_{r-1}|-\delta_{r}(\gamma,\zz)^{r-2}a_{r1}|\delta(\gamma,\zz) \ \ a_2 \ \ \cdots \ \ a_{r-1}|\\
    -\delta_{r}(\gamma,\zz)^{r-2}a_{r2}|a_1 \ \ \delta(\gamma,\zz) \ \ a_3 \ \ \cdots \ \ a_{r-1}|
    -\cdots-\delta_{r}(\gamma,\zz)^{r-2}a_{r(r-1)}|a_1 \ \ \cdots \ \ a_{r-2}\ \ \delta(\gamma,\zz)|.  
    \end{multline}
    On the other hand, observe, by the cofactor expansion of the determinant along the $r$-th row, the right hand side of \eqref{E:Eq19} is equal to
    \begin{multline}\label{E:Eq20}
    \delta_r(\gamma,\zz)^{r-2}\det\begin{pmatrix}
        a_{11}&\cdots &a_{1(r-1)}&\delta_1(\gamma,\zz)\\
        \vdots & &\vdots & \vdots\\
        a_{(r-1)1}&\cdots &a_{(r-1)(r-1)}&\delta_{r-1}(\gamma,\zz)\\
        a_{r1}&\cdots &a_{r(r-1)}&\delta_{r}(\gamma,\zz)
        \end{pmatrix}\\=\delta_r(\gamma,\zz)^{r-2}\det\begin{pmatrix}
        a_{11}&\cdots &a_{1(r-1)}&z_ra_{1r}\\
        \vdots & &\vdots & \vdots\\
        a_{(r-1)1}&\cdots &a_{(r-1)(r-1)}&z_ra_{(r-1)r}\\
        a_{r1}&\cdots &a_{r(r-1)}&z_ra_{rr}
        \end{pmatrix}.    
        \end{multline}
       Since, the right hand side of \eqref{E:Eq20} is equal to $\delta_r(\gamma,\zz)^{r-2}\det(\gamma)z_r$, combining \eqref{E:Eq18} and \eqref{E:Eq19}, we finally have
       \[
       \det((a_{ij}\delta_r(\gamma,\zz)-a_{rj}\delta_i(\gamma,\zz))_{i,j})=\delta_r(\gamma,\zz)^{r-2}\det(\gamma)z_r       \]
and thus it finishes the proof of the first equality. The second equality simply follows from the definition and the fact that $z_r=1$.
\end{proof}

Now we are ready to prove our proposition.
\begin{proposition}\label{P:1} For any $\gamma\in \GL_r(A)$, define $\mathfrak{C}^{\gamma}(\zz):=(\mathfrak{c}^{\gamma}_{jl}(\zz))_{j,l}\in \Mat_{r-1}(\CC_{\infty})$. The following statements hold.
	\begin{itemize}
    \item[(i)] $\mathfrak{C}^{\gamma^{-1}}(\gamma\cdot \zz)^{-1}=\mathfrak{C}^{\gamma}(\zz)$.
      \item[(ii)]  \[
	\sum_{\ell=1}^{r-1}\mathfrak{c}_{i\ell}^{\gamma}(\zz)c_{\ell r}^{\gamma^{-1}}=-j(\gamma,\zz)\det(\gamma)^{-1}c_{ir}^{\gamma}.
	\]		
         \item[(iii)] $\det(\mathfrak{C}^{\gamma}(\zz))=\det(\gamma)^{r-2}j(\gamma,\zz)$.
         \end{itemize}
\end{proposition}
\begin{proof} We start with proving the first assertion. Note that, since the formula $\gamma^{-1}=\det(\gamma)^{-1}(c_{ji}^{\gamma})_{i,j}$ implies $\gamma=(\gamma^{-1})^{-1}=(\det(\gamma^{-1}))^{-1}(c_{ji}^{\gamma^{-1}})_{i,j}$, we have 
\begin{equation}\label{E:Eq11}
    a_{ij}=\det(\gamma)c_{ji}^{\gamma^{-1}}.
\end{equation}
Moreover, setting $\partial(\gamma\cdot \zz):=\left(\partial_i\left(\frac{\delta_j(\gamma,\zz)}{j(\gamma,\zz)}\right)\right)_{i,j}$, we obtain
\begin{equation}\label{E:Eq12}
    \begin{split}
        \partial(\gamma\cdot\zz)&=\left(\partial_i\left(\frac{a_{j1}z_1+\cdots+a_{jr}z_r}{j(\gamma,\zz)}\right)\right)_{i,j} \\
&=\left(\frac{a_{ji}j(\gamma,\zz)-a_{ri}\delta_j(\gamma,\zz)}{j(\gamma,\zz)^2}\right)_{i,j}\\
&=\frac{\det(\gamma)}{j(\gamma,\zz)}\left(c_{ij}^{\gamma^{-1}}-c_{ir}^{\gamma^{-1}}\frac{\delta_j(\gamma\cdot\zz)}{j(\gamma,\zz)}\right)_{i,j}\\
&=\frac{\det(\gamma)}{j(\gamma,\zz)}(\mathfrak{c}_{ij}^{\gamma^{-1}}(\gamma\cdot\zz))_{i,j}\\
&=\frac{\det(\gamma)}{j(\gamma,\zz)}\mathfrak{C}^{\gamma^{-1}}(\gamma\cdot\zz).
\end{split}
\end{equation}
Here, the third equality follows from \eqref{E:Eq11} and the fourth equality follows from the fact that $j(\gamma,\zz)^{-1}\delta_j(\gamma,\zz)$ is the $j$-th entry, which we denote by $(\gamma\cdot \zz)_j$, of $\gamma\cdot \zz$. 

On the other hand, for any $\gamma,\nu\in \GL_r(A)$, by \cite[Lem. 3.1.3]{Bas14}, we know that 
\begin{equation}\label{E:Eq13}
j(\gamma\nu,\zz)=j(\gamma,\nu\cdot \zz)j(\nu,\zz).
\end{equation}
Applying the chain rule to $\gamma \nu=\gamma \circ \nu$, we obtain
\[
\frac{\det(\gamma\nu)}{j(\gamma\nu,\zz)}(\mathfrak{C}^{(\gamma\nu)^{-1}}(\gamma\nu\cdot \zz))^{\tr}=\frac{\det(\gamma)}{j(\gamma,\nu\cdot\zz)}(\mathfrak{C}^{\gamma^{-1}}(\gamma\cdot(\nu\cdot\zz)))^{\tr}\frac{\det(\nu)}{j(\nu,\zz)}(\mathfrak{C}^{\nu^{-1}}(\nu\cdot \zz))^{\tr},
\]
which, by \eqref{E:Eq13}, yields
\[
(\mathfrak{C}^{(\gamma\nu)^{-1}}(\gamma\nu\cdot \zz))^{\tr}=(\mathfrak{C}^{\gamma^{-1}}(\gamma\cdot(\nu\cdot\zz)))^{\tr}(\mathfrak{C}^{\nu^{-1}}(\nu\cdot \zz))^{\tr}.
\]
This implies
\begin{equation}\label{E:Eq14}
    \mathfrak{C}^{\gamma^{-1}}(\gamma\cdot\zz)\mathfrak{C}^{\gamma}(\zz)=\mathfrak{C}^{\Id_{r-1}}(\zz)=\Id_{r-1}\end{equation}
and hence we deduce the first assertion.

We now show the second assertion. Using \eqref{E:Eq13}, we have
\begin{equation}\label{E:Eq15}
    j(\gamma^{-1},\gamma\cdot \zz)=\frac{1}{j(\gamma,\zz)}.
\end{equation}
Note also that 
\[
j(\gamma^{-1},\zz)=\det(\gamma)^{-1}(c_{1r}^{\gamma}z_1+\cdots+c_{rr}^{\gamma}z_r)=\det(\gamma)^{-1}(c_{1r}^{\gamma}z_1+\cdots+c_{rr}^{\gamma}).
\]
We also obtain 
\begin{equation}\label{E:Eq16}
    \begin{split}(\partial_1(j(\gamma^{-1},\gamma\cdot\zz)),\dots,\partial_{r-1}(j(\gamma^{-1},\gamma\cdot\zz)))&=\left(\frac{c_{1r}^{\gamma}}{\det(\gamma)},\dots,\frac{c_{r-1,r}^{\gamma}}{\det(\gamma)}\right)(\partial(\gamma\cdot \zz))^{\tr}\\
    &=\frac{1}{j(\gamma,\zz)}(c_{1r}^{\gamma},\dots,c_{(r-1)r}^{\gamma})(\mathfrak{C}^{\gamma^{-1}}(\gamma\cdot\zz))^{\tr}    \end{split}
\end{equation}
where the first equality follows from the chain rule and the last equality follows from \eqref{E:Eq12}. Finally, we have
\begin{align*}
    (c_{1r}^{\gamma},\dots,c_{(r-1)r}^{\gamma})&=j(\gamma,\zz)(\partial_1(j(\gamma^{-1},\gamma\cdot\zz)),\dots,\partial_{r-1}(j(\gamma^{-1},\gamma\cdot\zz)))(\mathfrak{C}^{\gamma}(\zz))^{\tr} \\
    &=j(\gamma,\zz)\left(-\frac{a_{r1}}{j(\gamma,\zz)^2},\dots,-\frac{a_{r(r-1)}}{j(\gamma,\zz)^2}\right)(\mathfrak{C}^{\gamma}(\zz))^{\tr}\\
    &=-j(\gamma,\zz)^{-1}\det(\gamma)(c_{1r}^{\gamma^{-1}},\dots,c_{(r-1)r}^{\gamma^{-1}})(\mathfrak{C}^{\gamma}(\zz))^{\tr}.
    \end{align*}
    Here, the first equality follows from \eqref{E:Eq14} and \eqref{E:Eq16}, the second equality follows from \eqref{E:Eq15} and the last equality follows from \eqref{E:Eq11}. In other words, we have 
    \begin{equation}\label{E:Eq17}
        (c_{1r}^{\gamma^{-1}},\dots,c_{(r-1)r}^{\gamma^{-1}})(\mathfrak{C}^{\gamma}(\zz))^{\tr}=-j(\gamma,\zz)\det(\gamma)^{-1}(c_{1r}^{\gamma},\dots,c_{(r-1)r}^{\gamma}).        \end{equation}
Finally, comparing the $i$-th entry of both sides of \eqref{E:Eq17} yields the second assertion.

Lastly, observe that
\begin{align*}
    \det(\mathfrak{C}^{\gamma^{-1}}(\gamma\cdot\zz)^{\tr})&=\det(c_{ji}^{\gamma^{-1}}-c_{jr}^{\gamma^{-1}}(\gamma\cdot \zz)_{i})_{i,j}\\
    &=\det\left(\det(\gamma)^{-1}\left(a_{ij}-a_{rj}\frac{\delta_{i}(\gamma,\zz)}{j(\gamma,\zz)}\right)_{i,j}\right)\\
    &=\det(\gamma)^{1-r}j(\gamma,\zz)^{1-r}\det(a_{ij}j(\gamma,\zz)-a_{rj}\delta_i(\gamma,\zz))_{i,j}\\
    &=\det(\gamma)^{1-r}j(\gamma,\zz)^{1-r}j(\gamma,\zz)^{r-2}\det(\gamma)\\
    &=\det(\gamma)^{2-r}j(\gamma,\zz)^{-1}.
\end{align*}
Here, the second equality follows from \eqref{E:Eq11} and the third equality follows from Lemma \ref{L:det}. 
Then, the last assertion simply follows from the first assertion.
\end{proof}

We finish this section by describing functional equations of several rigid analytic functions, including our false Eisenstein series defined in \S1 and partial derivatives of Drinfeld modular forms, under the action of $\GL_r(A)$ on $\Omega^r$.

\begin{lemma}\label{L:1} Let $f\in M_{k}^m(\Gamma)$ for some congruence subgroup $\Gamma \leqslant \GL_r(A)$, $k\in \ZZ_{\geq 0}$ and $m\in \ZZ/(q-1)\ZZ$. For any $1\leq i,j \leq r-1$, the following identities hold.
	\begin{itemize}
		\item[(i)]  For any $\gamma\in \GL_{r}(A)$,  $E_r^{[i]}(\gamma\cdot \zz)=j(\gamma,\zz)\det(\gamma)^{-1}\Big(\sum_{l=1}^{r-1}E_r^{[l]}(\zz)\mathfrak{c}^{\gamma}_{il}(\zz)+c^{\gamma}_{ir}\Big)$.
		\item[(ii)] For any $\gamma\in \Gamma$, we have
		\[
		 \partial_{i}(f(\gamma\cdot\zz))=j(\gamma,\zz)^{k+1}\det(\gamma)^{-m-1}\Big(\sum_{l=1}^{r-1}\partial_{l}(f(\zz))\mathfrak{c}^{\gamma}_{il}(\zz)-kc^{\gamma}_{ir}f(\zz)\Big).
		 \]
		\item[(iii)] Consider the operator $\mathcal{D}_{i,k}$ given by  $\mathcal{D}_{i,k}(f)=\partial_{i}(f)+kE_r^{[i]}f$. For any $\gamma\in \Gamma$, we have
		\[
		\mathcal{D}_{i,k}(f)(\gamma\cdot\zz)=j(\gamma,\zz)^{k+1}\det(\gamma)^{-m-1}\sum_{l=1}^{r-1}\mathcal{D}_{l,k}(f)(\zz)\mathfrak{c}^{\gamma}_{il}(\zz).
		\]
	\end{itemize}	
\end{lemma}
\begin{proof} The first two identities basically follow from the same method used in \cite[Sec. 4]{CG21}. Since, by assumption, we have 
\[
f(\gamma\cdot \zz)=j(\gamma,\zz)^k\det(\gamma)^{-m}f(\zz),
\]
the third part is a consequence of (i) and (ii) as well as the fact that $\partial_i$ is a $\CC_{\infty}$-linear derivation on the set of rigid analytic functions on $\Omega^r$. 
\end{proof}

\subsection{Proof of Theorem \ref{T:I2}}
In this subsection, we provide a proof for Theorem \ref{T:I2}. For this goal, in what follows, we first need some preliminary lemmas. We would like to thank the anonymous referee once again for informing us about a simpler approach to prove Proposition \ref{P:Hol} and we provide its details in what follows. We refer the reader to the previous version of the manuscript \cite[\S5.1 and Prop.~5.23]{CG22} for a different idea to prove Proposition \ref{P:Hol}.

Let $\zz\in \Omega^r$ and $\{y_1,\dots,y_r\}$ be a successive minimum basis of $A^{r}\zz$ in the sense of Gekeler \cite[\S3]{Gek17b}. We further set $\delta:=\min\{|y_i| \ \ | \ \ i=1,\dots,r\}>0$. Let $\mathfrak{B}\in \GL_r(A)$ be such that 
\begin{equation}\label{E:Eq1}
    (z_1,\dots,z_r)=(y_1,\dots,y_r)\mathfrak{B}.
\end{equation}
For any $M\in \Mat_{\ell\times n}(\mathbb{C}_{\infty})$, we let $|M|$ be the maximum of the norm of entries of $M$. 

\begin{lemma}\label{L:bound} For any $k\in \mathbb{Z}_{\geq 1}$, $1\leq j \leq r$ and any tuple $\mathfrak{a}=(a_1,\dots,a_r)\in A^{r}\setminus \{(0,\dots,0)\}$, we have 
\[
\left|\frac{a_j}{(a_1z_1+\dots+a_rz_r)^k}  \right|\leq \frac{|\mathfrak{B}^{-1}|}{\delta^k|\mathfrak{B}\mathfrak{a}^{\tr}|^{k-1}}.
\]
\end{lemma}
\begin{proof} We set $\mathfrak{c}:=(c_1,\dots,c_r)^{\tr}:=\mathfrak{B}\mathfrak{a}^{\tr}$. Since $\mathfrak{B}\in \GL_r(A)$ and $\mathfrak{a}\neq (0,\dots,0)$, we have $\mathfrak{c}\neq (0,\dots,0)^{\tr}$. Moreover, by \eqref{E:Eq1}, we have 
\[
a_1z_1+\dots+a_rz_r=(y_1,\dots,y_r)\mathfrak{B}\mathfrak{a}^{\tr}=(y_1,\dots,y_r)\mathfrak{c}=\sum_{i=1}^{r}c_iy_i.
\]
    Thus, we have 
    \[
|a_1z_1+\dots+a_rz_r|=\max\{|y_i||c_i| \ \ | \ \ i=1,\dots,r\}\geq \delta \max_{1\leq i \leq r}|c_i|=\delta |\mathfrak{c}|
    \]
    where the first equality follows form \cite[Prop. 3.1(iii)]{Gek17b}. On the other hand, since $|\mathfrak{a}|=|\mathfrak{B}^{-1}\mathfrak{c}|\leq |\mathfrak{B}^{-1}||\mathfrak{c}| $, we further obtain
    \[
    \left| \frac{a_j}{(a_1z_1+\dots+a_rz_r)^k} \right|\leq \frac{|\mathfrak{B}^{-1}\mathfrak{c}|}{(\delta |\mathfrak{c}|)^k}\leq \frac{|\mathfrak{B}^{-1}||\mathfrak{c}|}{\delta^k |\mathfrak{c}|^{k}}=\frac{|\mathfrak{B}^{-1}|}{\delta^k|\mathfrak{c}|^{k-1}}= \frac{|\mathfrak{B}^{-1}|}{\delta^k |\mathfrak{B}\mathfrak{a}^{\tr}|^{k-1}}  \]
    as desired.
\end{proof}

\begin{lemma}\label{L:holEisenstein} For any integer $k\geq 2$ and $1\leq j \leq r$, the infinite sum 
\[
\sum_{\substack{(a_1,\dots,a_r)\in A^r\\(a_1,\dots,a_r)\neq (0,\dots,0)}}\frac{a_j}{(a_1z_1+\dots+a_rz_r)^k}\]
converges.
\end{lemma}
\begin{proof} Let $\mathfrak{a}$ be as in Lemma \ref{L:bound}. Since $|\mathfrak{a}|\to \infty$ if and only if $|\mathfrak{B}\mathfrak{a}^{\tr}|\to \infty$, the lemma follows from Lemma \ref{L:bound}.
    
\end{proof}

For any $\zz=(z_1,z_2,\dots,z_r)^{\tr}\in \Omega^r$, we set  
\[
|\zz|_{\text{im}}:=\mathrm{inf}\{\inorm{z_1-\mathfrak{a}}:\mathfrak{a}=a_2z_2+\dots+a_rz_r, \ \  a_2,\dots,a_r\in K_{\infty}\}.
\]
For $\ww\in \Omega^{r-1}$, let $\mathfrak{I}_{\ww}$ be the set of elements of the form $\zz=(z_1,\ww)\in \Omega^r$ satisfying $\inorm{z_1}=|\zz|_{\text{im}}$. 

\begin{lemma} \label{L:bounds2} Let $1\leq j \leq r-1$ and $\ww\in \Omega^{r-1}$. The following statements hold.
\begin{itemize}
    \item[(i)] Let $\mathfrak{n}\in A\setminus\{0\}$. We have 
    \[
    \lim_{\substack{\zz\in \mathfrak{I}_{\ww}\\ |\zz|_{\infty}\to\infty}}\partial_j(u_{\Gamma(\mathfrak{n})}(\zz))=0.
    \]
    \item[(ii)] For $j\geq 2$, the absolute value $|E_r^{[j]}(\zz)|$ is bounded independently of $z_1$ when $\zz\in \mathfrak{I}_{\ww}$ and $ |\zz|_{\infty}\to\infty$.
    \end{itemize}
\end{lemma}
\begin{proof}  Recall from \eqref{E:cong} that, up to a scalar multiple, $u_{\Gamma(\mathfrak{n})}(\zz)$ is obtained by replacing $z_1$ in $u(\zz)$ by $z_1/\mathfrak{n}$. Furthermore, it is easy to see that $(z_1,\ww)\in \mathfrak{I}_{\ww}$ if and only if $(z_1/\mathfrak{n},\ww)\in \mathfrak{I}_{\ww}$. Thus, to prove the first assertion, it suffices to show that 
\begin{equation}\label{E:Eq2}
\lim_{\substack{\zz\in \mathfrak{I}_{\ww}\\ |\zz|_{\infty}\to\infty}}\partial_j(u(\zz))=0.
\end{equation}
We first analyze the case $j=1$. By \cite[Rem. 3.22]{CG21}, when $(z_1,\ww)\in \mathfrak{I}_{\ww}$, we have 
\begin{equation}\label{E:Eq3}
    |\exp_{A^{r-1}\ww}(z_1)|=\left|z_1\prod_{\lambda\in A^{r-1}\ww\setminus\{0\}}\left(1-\frac{z_1}{\lambda}\right)\right|=|z_1|\prod_{0<|\lambda|<|z_1|}\frac{|z_1|}{|\lambda|}.
    \end{equation}
    Observe that $\partial_1(u(\zz))$ is a constant multiple of $u(\zz)^{2}$. Moreover, since the conditions $\zz\in \mathfrak{I}_{\ww}$ and $|\zz|_{\infty}\to \infty$ imply that $|z_1|\to \infty$, by \eqref{E:Eq3}, we have 
    \begin{equation}\label{E:Eq4}
        \lim_{\substack{\zz\in \mathfrak{I}_{\ww}\\ |\zz|_{\infty}\to\infty}}u(\zz)=0
        \end{equation}
    and hence, when $j=1$, we obtain \eqref{E:Eq2}.

    We now assume that $j\geq 2$. We claim that the norm of $\partial_j(u(\zz))/u(\zz)$ is bounded independent of $z_1$. Observe that, by \eqref{E:Eq4}, the claim then implies \eqref{E:Eq2} for $j\geq 2$. Hence, we are reduced to showing our claim. 

We write $\ww=(w_1,\dots,w_{r-1})^{\tr}\in \Omega^{r-1}$. Then we have 
    \begin{equation}\label{E:Eq5}
    \begin{split}
        \frac{\partial_j(u(\zz))}{u(\zz)}&=\exp_{A^{r-1}\ww}(z_1)\partial_j\left(\frac{1}{\exp_{A^{r-1}\ww}(z_1)}\right)\\
        &=\frac{\partial_j(\exp_{A^{r-1}\ww}(z_1))}{\exp_{A^{r-1}\ww}(z_1)}\\
        &=\sum_{\lambda\in A^{r-1}\ww\setminus\{0\}}\frac{\frac{z_1\partial_j(\lambda)}{\lambda^2}}{1-\frac{z_1}{\lambda}}\\
        &=\sum_{\substack{(a_1,\dots,a_r)\in A^{r-1}\\(a_1,\dots,a_{r-1})\neq (0,\dots,0)}}\frac{a_{j-1}}{(a_1w_1+\dots+a_{r-1}w_{r-1})^2}\frac{z_1}{1-\frac{z_1}{a_1w_1+\dots+a_{r-1}w_{r-1}}}.
  \end{split}
  \end{equation}
Here, the second equality follows from the chain rule and the third equality follows from the logarithmic derivative of $\exp_{A^{r-1}\ww}(z_1)$ with respect to the variable $w_j$. 

Let $\lambda=a_1w_1+\dots+a_{r-1}w_{r-1}\in A^{r-1}\ww$. Since, $|\zz|_{\text{im}}=|z_1|$, by using the same analysis in \cite[Rem. 3.22]{CG21}, we obtain 
\begin{equation}\label{E:Eq6}
\left|\frac{1}{1-\frac{z_1}{\lambda}}\right|=\begin{cases}
    \frac{|\lambda|}{|z_1|} & \text{ if }|\lambda|<|z_1| \\
    1 & \text{ if }|\lambda|\geq |z_1|.
\end{cases}
\end{equation}
Moreover, using \eqref{E:Eq6}, we obtain
\[
\left|\frac{z_1}{1-\frac{z_1}{\lambda}}\right|=\begin{cases}
    |\lambda| & \text{ if }|\lambda|<|z_1| \\
    |z_1| & \text{ if }|\lambda|\geq |z_1|,
\end{cases}
\]
  implying that
  \begin{equation}\label{E:Eq7}
      \left|\frac{z_1}{1-\frac{z_1}{\lambda}}\right|=\left|\frac{z_1}{1-\frac{z_1}{a_1w_1+\dots+a_{r-1}w_{r-1}}}\right|\leq |\lambda|=|a_1w_1+\dots+a_{r-1}w_{r-1}|.
      \end{equation}
Thus, by Lemma \ref{L:bound} and \eqref{E:Eq7}, we obtain
\begin{equation}\label{E:Eq8}
\left|\frac{a_{j-1}}{(a_1w_1+\dots+a_{r-1}w_{r-1})^2}\frac{z_1}{1-\frac{z_1}{a_1w_1+\dots+a_{r-1}w_{r-1}}}\right|\leq \frac{|a_{j-1}|}{|a_1w_1+\dots+a_{r-1}w_{r-1}|}\leq \frac{|\mathfrak{B}^{-1}|}{\delta}.
\end{equation}
      Now combining \eqref{E:Eq5} and \eqref{E:Eq8} yields the claim, hence finishing the proof of the first assertion.

      We now show the second assertion. For any sufficiently large $z_1$ and $\zz=(z_1,\ww)\in \mathfrak{I}_{\ww}$, we write
      \[
      h_r(\zz)=\sum_{i=1}^{\infty}c_i(\ww)u(\zz)^i
      \]
      for some uniquely determined rigid analytic functions $c_i:\Omega^{r-1}\to \mathbb{C}_{\infty}$. By \cite[Cor. 6.3]{BB17}, we know that $c_1$ is a constant multiple of the Gekeler's $h$-function $h_{r-1}$ which is nowhere vanishing on $\Omega^{r-1}$. Then, by \eqref{E:Eq4}, we have 
      \[
       \lim_{\substack{\zz\in \mathfrak{I}_{\ww}\\ |\zz|_{\infty}\to\infty}}\frac{h_r(\zz)}{u(\zz)}=c_1(\ww)\neq 0
       \]
       implying that 
       \begin{equation}\label{E:Eq9}
         \lim_{\substack{\zz\in \mathfrak{I}_{\ww}\\ |\zz|_{\infty}\to\infty}}\frac{u(\zz)}{h_r(\zz)}=c_1(\ww)^{-1}.           
       \end{equation}
      On the other hand, by \eqref{E:hfunc} and the chain rule, we obtain
       \begin{multline*}
       E_r^{[j]}(\zz)=\frac{\partial_j(g_r(\zz))}{g_r(\zz)}=\frac{\partial_j(h_r^{q-1}(\zz))}{h_r^{q-1}(\zz)}=-\frac{\partial_j(h_r(\zz))}{h_r(\zz)}\\
       =-\frac{1}{h_r(\zz)}\left(\sum_{i=1}^{\infty}\partial_j(c_i(\ww))u(\zz)^i+ \partial_j(u(\zz))\sum_{i=1}^{\infty}c_i(\ww)iu(\zz)^{i-1} \right).
       \end{multline*}
       Finally, since the norm of $\partial_j(u(\zz))/u(\zz)$ is bounded independent of $z_1$, the desired statement follows from \eqref{E:Eq4} and \eqref{E:Eq9}.
\end{proof}

 For each $1\leq \ell \leq r-1$, let $f_\ell\in M_{k_\ell}^{m_\ell}(\Gamma_\ell)$ for some congruence subgroup $\Gamma_\ell \leqslant \GL_r(A)$, $k_\ell\in \ZZ_{\geq 0}$ and $m_\ell\in \ZZ/(q-1)\ZZ$. We further let  $\Gamma':=\cap_{\ell=1}^{r-1}\Gamma_\ell$. Set $\mathfrak{k}:= k_1+\dots+k_{r-1}+r$ and $\mathfrak{m}:=m_1+\dots+m_{r-1}+1$. Recall the operator $\D_{(k_1,\dots,k_{r-1})}$ on $M_{k_1}^{m_1}(\Gamma_1)\times \cdots \times M_{k_{r-1}}^{m_{r-1}}(\Gamma_{r-1})$ defined by 
	\begin{equation}\label{E:operator}
	\D_{(k_1,\dots,k_{r-1})}(f_1,\dots,f_{r-1})=\det \begin{pmatrix}
	\mathcal{D}_{1,k_1}(f_1)&\cdots &\mathcal{D}_{r-1,k_1}(f_1)\\
	\vdots & & \vdots  \\
	\mathcal{D}_{1,k_{r-1}}(f_{r-1})& \cdots &\mathcal{D}_{r-1,k_{r-1}}(f_{r-1})
	\end{pmatrix}.
	\end{equation}
Furthermore, we set $\mathcal{D}^{(\ell,1)}$ to be the $(\ell,1)$-cofactor of the matrix \[
\begin{pmatrix}
	\mathcal{D}_{1,k_1}(f_1)&\cdots &\mathcal{D}_{r-1,k_1}(f_1)\\
	\vdots & & \vdots  \\
	\mathcal{D}_{1,k_{r-1}}(f_{r-1})& \cdots &\mathcal{D}_{r-1,k_{r-1}}(f_{r-1})
	\end{pmatrix}.
 \] 
 
Now we are ready to prove our next proposition.

\begin{proposition}\label{P:Hol} For each $1\leq \ell \leq r-1$, we have
	\[
	\lim_{\substack{\zz\in \mathfrak{I}_{\ww}\\ |\zz|_{\infty}\to\infty}}\mathcal{D}_{1,k_{\ell}}(f_{\ell})(\zz)\mathcal{D}^{(\ell,1)}(\zz)=0.
	\]
 In particular, 
 \[
\lim_{\substack{\zz\in \mathfrak{I}_{\ww}\\ |\zz|_{\infty}\to\infty}}\D_{(k_1,\dots,k_{r-1})}(f_1,\dots,f_{r-1})(\zz)=0.
\]
\end{proposition}
\begin{proof} There exists a polynomial $\mathfrak{n}\in A\setminus \{0\}$ such that $\Gamma(\mathfrak{n})$ is contained in $\Gamma'$. Moreover, for each $i\in \ZZ_{\geq 0}$, there exists a unique rigid analytic function $\mathfrak{F}_{i,\ell}:\Omega^{r-1}\to \CC_{\infty}$  such that 
	\[
	f_{\ell}(\zz)=\sum_{i=0}^{\infty}\mathfrak{F}_{i,\ell}(\ww)u_{\Gamma(\mathfrak{n})}(\zz)^{i}
	\]
	whenever $\zz$ is in some neighborhood of infinity. We note, by the infinite sum expansion of $E_r^{[1]}$ given in \cite[(4.2)]{CG21}, that $E_r^{[1]}$ has a $u=u_{\Gamma(1)}$-expansion with no constant term. Moreover, $\partial_1(u_{\Gamma(\mathfrak{n})}(\zz))$ is a constant
multiple of $u_{\Gamma(\mathfrak{n})}(\zz)^2$. On the other hand, since $u(\zz)$ can be written as an infinite series of $u_{\Gamma(\mathfrak{n})}(\zz)$ with no constant term, we deduce that for $k\in \mathbb{Z}_{\geq 1}$, $\mathcal{D}_{1,k_{\ell}}(f_{\ell})$ has a $u_{\Gamma(\mathfrak{n})}$-expansion with no constant term. Furthermore, letting $2\leq j \leq r-1$ and using  the definition of the operator $\mathcal{D}_{j,k_{\ell}}$, we obtain
\begin{equation}\label{E:partialanalysis}
	\mathcal{D}_{j,k_\ell}(f_\ell(\zz))=\sum_{i=0}^{\infty} \partial_{j}(\mathfrak{F}_{i,\ell})(\ww)u_{\Gamma(\mathfrak{n})}(\zz)^{i}+\sum_{i=1}^{\infty} \mathfrak{F}_{i,\ell}(\ww)\partial_{j}(u_{\Gamma(\mathfrak{n})}^{i})(\zz)+k_{\ell}f_{\ell}(\zz)E_r^{[j]}(\zz).
 \end{equation}
 Then the first part of the proposition follows from Lemma \ref{L:bounds2}, \eqref{E:Eq4} and \eqref{E:partialanalysis}. The second assertion simply follows from the first assertion and the cofactor expansion of determinants along the first column.
\end{proof}

We are now ready to prove the main result of this section.

\begin{theorem}\label{T:RS} Let $\D_{(k_1,\dots,k_{r-1})}(f_1,\dots,f_{r-1})$ be the operator defined as in \eqref{E:operator}. The following statements hold.
	\begin{itemize}
		\item[(i)] $\D_{(k_1,\dots,k_{r-1})}$ is a $\CC_{\infty}$-multi-linear derivation. In other words, for any $f_i,\tilde{f}_i\in M_{k_i}^{m_i}(\Gamma_i)$, we have
		\begin{multline*}
		\D_{(k_1,\dots,k_{r-1})}(f_1,\dots,f_{i-1},f_i+\tilde{f}_i,f_{i+1},\dots,f_{r-1})\\=\D_{(k_1,\dots,k_{r-1})}(f_1,\dots,f_{i-1},f_i,f_{i+1},\dots,f_{r-1})+\D_{(k_1,\dots,k_{r-1})}(f_1,\dots,f_{i-1},\tilde{f}_i,f_{i+1},\dots,f_{r-1})
		\end{multline*}
		and for any $f\in M_{k_i}^{m_i}(\Gamma_i)$ and $g\in M_{k_j}^{m_j}(\Gamma_i)$, we have	\begin{multline*}
		\D_{(k_1,\dots,k_i+k_j,\dots,k_{r-1})}(f_1,\dots,f_{i-1},fg,f_{i+1},\dots,f_{r-1})\\=g\D_{(k_1,\dots,k_i,\dots,k_{r-1})}(f_1,\dots,f_{i-1},f,f_{i+1},\dots,f_{r-1})\\+f\D_{(k_1,\dots,k_j,\dots,k_{r-1})}(f_1,\dots,f_{i-1},g,f_{i+1},\dots,f_{r-1}).
		\end{multline*}
		\item[(ii)]  $\D_{(k_1,\dots,k_{r-1})}(f_1,\dots,f_{r-1})\in M_{\mathfrak{k}}^{\mathfrak{m}}(\Gamma')$.
	\end{itemize}
\end{theorem}
\begin{proof} To prove part (i), by Lemma \ref{L:1}(iii), observe for $1\leq \mu \leq r-1$ that 
	$
	\mathcal{D}_{\mu,k_i}(f_i+\tilde{f}_i)=\mathcal{D}_{\mu,k_i}(f_i)+\mathcal{D}_{\mu,k_i}(\tilde{f}_i)	
	$ 
	and
	$
	\mathcal{D}_{\mu,k_i+k_j}(fg)=g\mathcal{D}_{\mu,k_i}(f)+f\mathcal{D}_{\mu,k_j}(g).	
	$ 
	Thus the desired equalities follow from the multi-linearity in rows of the determinant function. 
	
	We now prove the second part. Note that using Lemma \ref{L:1}(iii) and Proposition \ref{P:1}(i), for any $\gamma \in \Gamma'$, we obtain
	\begin{align*}
	\D_{(k_1,\dots,k_{r-1})}(f_1,\dots,f_{r-1})(\gamma \cdot \zz)&=\det(\gamma)^{-\mathfrak{m}-r+2}j(\gamma,\zz)^{\mathfrak{k}-1}\det(\mathfrak{C}^{\gamma}(\zz))\D_{(k_1,\dots,k_{r-1})}(f_1,\dots,f_{r-1})(\zz)\\
	&=\det(\gamma)^{-\mathfrak{m}}j(\gamma,\zz)^{\mathfrak{k}}\D_{(k_1,\dots,k_{r-1})}(f_1,\dots,f_{r-1})(\zz).
	\end{align*} 
	 Therefore $\D_r(f_1,\dots,f_{r-1})(\zz)$ is a weak modular form of weight $\mathfrak{k}$ and type $\mathfrak{m}$ for $\Gamma'$. 
	
	Let $\nu=(a_{ij})_{ij}\in\GL_r(A)$ and consider the congruence subgroup $\tilde{\Gamma}:=\nu^{-1}\Gamma'\nu\leq \GL_{r}(A)$. There exists a non-constant element $\mathfrak{n}\in A$ such that $\Gamma(\mathfrak{n})\subseteq \tilde{\Gamma}.$  We claim that the function 
	\[
	\D_{(k_1,\dots,k_{r-1})}(f_1,\dots,f_{r-1})|_{\mathfrak{k},\mathfrak{m}}[\nu](\zz)=j(\nu,\zz)^{-\mathfrak{k}}\det(\nu)^{\mathfrak{m}}\D_{(k_1,\dots,k_{r-1})}(f_1,\dots,f_{r-1})(\nu\cdot \zz)
	\]
	is holomorphic at infinity with respect to $\Gamma(\mathfrak{n})_{\iota}$. 
	
	Firstly, for any $1\leq j \leq r-1$, set 
\begin{equation}\label{E:CMF}
f_j(\nu\cdot \zz)=j(\nu,\zz)^{k_j}\det(\nu)^{-m_j}\mathfrak{g}_j(\zz)
\end{equation}
where $\mathfrak{g}_j$ is a Drinfeld modular form of weight $k_j$ and type $m_j$ for $\nu^{-1}\Gamma_j\nu\leq  \GL_r(A)$ (see \cite[Prop. 6.6]{BBP18}). Applying $\partial_i$ to both sides of \eqref{E:CMF} for $1\leq i \leq r-1$ and using the chain rule, we obtain
\begin{multline*}
\partial_i(f_j(\nu\cdot \zz))\\
    =\partial_1(f_j)(\nu\cdot \zz)\left(\frac{a_{1i}}{j(\nu,\zz)}-\frac{a_{ri}\delta_1(\nu,\zz)}{j(\nu,\zz)^2}\right)+\cdots +\partial_{r-1}(f_j)(\nu\cdot \zz)\left(\frac{a_{(r-1)i}}{j(\nu,\zz)}-\frac{a_{ri}\delta_{r-1}(\nu,\zz)}{j(\nu,\zz)^2}\right)\\
    =j(\nu,\zz)^{k_j}\det(\nu)^{-m_j}\partial_i(\mathfrak{g}_j)(\zz)+k_ja_{ri}j(\nu,\zz)^{k_j-1}\det(\nu)^{-m_j}\mathfrak{g}_j(\zz).
\end{multline*}
Recall that the $j$-th entry $(\nu\cdot \zz)_j$ of $\nu\cdot \zz$ is $j(\nu,\zz)^{-1}\delta_j(\nu,\zz)$. Thus, using \eqref{E:Eq11} and multiplying both sides of above with $j(\nu,\zz)$, we obtain
\begin{multline*}
    \partial_1(f_j)(\nu\cdot \zz)(c^{\nu^{-1}}_{i1}-c^{\nu^{-1}}_{ir}(\nu\cdot \zz)_1)+\cdots  +\partial_{r-1}(f_j)(\nu\cdot \zz)(c^{\nu^{-1}}_{i(r-1)}-c^{\nu^{-1}}_{ir}(\nu\cdot \zz)_{r-1})\\
    =(\partial_1(f_j)(\nu\cdot \zz),\dots ,\partial_{r-1}(f_j)(\nu\cdot \zz))\begin{pmatrix}\mathfrak{c}^{\nu^{-1}}_{i1}(\nu\cdot \zz)\\\vdots\\
    \mathfrak{c}^{\nu^{-1}}_{i(r-1)}(\nu\cdot \zz)\end{pmatrix}\\
    =j(\nu,\zz)^{k_j+1}\det(\nu)^{-m_j-1}\partial_i(\mathfrak{g}_j)(\zz)+k_jj(\nu,\zz)^{k_j}\det(\nu)^{-m_j}c^{\nu^{-1}}_{ir}\mathfrak{g}_j(\zz).
\end{multline*}
That is, using \eqref{E:CMF}, we now have
\begin{multline*}
    (\partial_1(f_j)(\nu\cdot \zz),\dots ,\partial_{r-1}(f_j)(\nu\cdot \zz))(\mathfrak{C}^{\nu^{-1}}(\nu\cdot \zz))^{\tr}\\
    =(j(\nu,\zz)^{k_j+1}\det(\nu)^{-m_j-1}\partial_1(\mathfrak{g}_j)(\zz)+k_jc^{\nu^{-1}}_{1r}f_j(\nu\cdot\zz)\\
    ,\dots,j(\nu,\zz)^{k_j+1}\det(\nu)^{-m_j-1}\partial_{r-1}(\mathfrak{g}_j)(\zz)+k_jc^{\nu^{-1}}_{(r-1)r}f_j(\nu\cdot\zz)).
\end{multline*}
Using Proposition \ref{P:1}(i), we further obtain
\begin{multline*}
    (\partial_1(f_j)(\nu\cdot \zz),\dots ,\partial_{r-1}(f_j)(\nu\cdot \zz))\\
    =(j(\nu,\zz)^{k_j+1}\det(\nu)^{-m_j-1}\partial_1(\mathfrak{g}_j)(\zz)+k_jc^{\nu^{-1}}_{1r}f_j(\nu\cdot\zz)\\
    ,\dots,j(\nu,\zz)^{k_j+1}\det(\nu)^{-m_j-1}\partial_{r-1}(\mathfrak{g}_j)(\zz)+k_jc^{\nu^{-1}}_{(r-1)r}f_j(\nu\cdot\zz))(\mathfrak{C}^{\nu}(\zz))^{\tr}.
\end{multline*}
Expanding out above and using Proposition \ref{P:1}(ii), we finally have
\begin{equation}\label{E:DER3}
\begin{split}
\partial_i(f_j)(\nu\cdot \zz)&=j(\nu,\zz)^{k_j}\det(\nu)^{-m_j}\left(j(\nu,\zz)\det(\nu)^{-1}\sum_{l=1}^{r-1}\partial_l(\mathfrak{g}_j)(\zz)\mathfrak{c}_{il}^{\nu}(\zz)\right)+k_jf_j(\nu\cdot\zz)\sum_{k=1}^{r-1}\mathfrak{c}_{ik}^{\nu}(\zz)c_{kr}^{\nu^{-1}}\\
&=j(\nu,\zz)^{k_j+1}\det(\nu)^{-m_j-1}\sum_{l=1}^{r-1}\partial_l(\mathfrak{g}_j)(\zz)\mathfrak{c}_{il}^{\nu}(\zz)-k_jj(\nu,\zz)\det(\nu)^{-1}f_j(\nu\cdot \zz)c_{ir}^{\nu}.
\end{split}
\end{equation}
Thus, Lemma \ref{L:1}(i) and \eqref{E:DER3} imply that 
\begin{equation}\label{E:Der4}
\mathcal{D}_{i,k_j}(f_j)(\nu\cdot \zz)=(\partial_i(f_j)+k_jf_jE_r^{[i]})(\nu\cdot \zz)=j(\nu,\zz)^{k_j+1}\det(\nu)^{-m_j-1}\sum_{l=1}^{r-1}\mathcal{D}_{l,k_j}(\mathfrak{g}_j(\zz))\mathfrak{c}_{il}^{\nu}(\zz).
\end{equation}
Hence, by Proposition \ref{P:1}(iii) and \eqref{E:Der4}, we obtain  
\begin{equation}\label{E:cuspanalysis}
\D_{(k_1,\dots,k_{r-1})}(f_1,\dots,f_{r-1})|_{\mathfrak{k},\mathfrak{m}}[\nu]=\det \begin{pmatrix}
\mathcal{D}_{1,k_1}(\mathfrak{g}_1)&\cdots &\mathcal{D}_{r-1,k_1}(\mathfrak{g}_1)\\
\vdots & & \vdots  \\
\mathcal{D}_{1,k_{r-1}}(\mathfrak{g}_{r-1})& \cdots &\mathcal{D}_{r-1,k_{r-1}}(\mathfrak{g}_{r-1})
\end{pmatrix}.
\end{equation}
 Since, for each $1\leq j \leq r-1$, $\mathfrak{g}_j$ is a Drinfeld modular form for $\Gamma(\mathfrak{n})$, the above discussion implies that $\D_{(k_1,\dots,k_{r-1})}(f_1,\dots,f_{r-1})|_{\mathfrak{k},\mathfrak{m}}[\nu]$ is a weak modular form of weight $\mathfrak{k}$ and type $\mathfrak{m}$ for $\Gamma(\mathfrak{n})$. Therefore, for any $\zz\in \Omega^r$ in some neighborhood of infinity, there exists $n_0\in \ZZ$ and for each $n\geq n_0$, a unique rigid analytic function $F_n:\Omega^{r-1}\to \CC_{\infty}$ such that 
\[
\D_{(k_1,\dots,k_{r-1})}(f_1,\dots,f_{r-1})|_{\mathfrak{k},\mathfrak{m}}[\nu](\zz)=\sum_{n\geq n_0}F_n(\ww)u_{\Gamma(\mathfrak{n})}(\zz)^{n}.
\]
Assume that $n_0<0$. Then we have 
\begin{multline*}
0= \lim_{\substack{\zz\in \mathfrak{I}_{\ww}\\ |\zz|_{\infty}\to\infty}}\D_{(k_1,\dots,k_{r-1})}(f_1,\dots,f_{r-1})|_{\mathfrak{k},\mathfrak{m}}[\nu](\zz)=\lim_{\substack{\zz\in \mathfrak{I}_{\ww}\\ |\zz|_{\infty}\to\infty}}\sum_{n\geq n_0}F_n(\ww)u_{\Gamma(\mathfrak{n})}(\zz)^{n}\to \infty
\end{multline*}
where the equality on the left hand side follows from \eqref{E:cuspanalysis} and Proposition \ref{P:Hol}, and the fact that the right hand side approaches to the infinity follows from our assumption on $n_0$. This gives us a contradiction and therefore $n_0$ must be a non-negative integer. Thus, $\D_{(k_1,\dots,k_{r-1})}(f_1,\dots,f_{r-1})|_{\mathfrak{k},\mathfrak{m}}[\nu]$ is holomorphic at infinity with respect to $\Gamma(\mathfrak{n})_{\iota}$. Hence $\D_{(k_1,\dots,k_{r-1})}(f_1,\dots,f_{r-1})$ is a Drinfeld modular form for $\Gamma(\mathfrak{n})$. Finally, using \cite[Prop. 5.14]{BBP18}, we see that $\D_{(k_1,\dots,k_{r-1})}(f_1,\dots,f_{r-1})$ is a Drinfeld modular form with respect to $\tilde{\Gamma}$ and hence so is with respect to $\Gamma'$. 
\end{proof}
\begin{remark} We highlight the work of Pellarin \cite[Sec. 6.1]{Pel05} providing a multi-linear differential operator by using the partial derivatives of Hilbert modular forms which may be compared with our construction.
\end{remark}

\section{Partial derivatives of Drinfeld modular forms}
Our first goal in this section is to prove Theorem \ref{T:GM2} which establish formulas for the partial derivatives of coefficient forms described in Example \ref{Ex:1}(ii) and the $h$-function $h_r$. Later on, we  provide an alternative description for Gekeler's $h$-function $h_r$ in Proposition \ref{P:hfunc}. Lastly, we will describe a $\mathbb{C}_{\infty}$-algebra $\mathcal{M}_r$ invariant under the partial derivatives and provide an alternative description for it.
\subsection{The Tate algebra and Anderson generating functions}
Let $t$ be a variable over $\CC_{\infty}$. We define \textit{the Tate algebra} $\mathbb{T}$ as the subring of power series in $\CC_{\infty}[[t]]$ satisfying a certain condition:
\[
\mathbb{T}:=\left\{\sum_{i\geq 0} c_it^i\in \CC_{\infty}[[t]] \ \ | \ \ \inorm{c_i}\to 0 \text{ as } i\to \infty\right\}.
\]
For any $g=\sum_{j\geq 0}g_jt^j\in \TT$, we define \textit{the Gauss norm of $g$} by 
\[
\dnorm{g}:=\sup\{ \inorm{c_j} \ \ | \ \ j\in \ZZ_{\geq 0}\}.
\]
Note that $\TT$, equipped with the Gauss norm, forms a Banach algebra. Furthermore, for any $i\in \ZZ$ and $g\in \TT$ as above, we define \textit{the $i$-th fold twist of $g$} to be $g^{(i)}:=\sum_{j\geq 0}g_j^{q^i}t^j\in \TT$.

We now briefly discuss the theory of Drinfeld $A[t]$-modules over $\TT$ introduced by Angl\`{e}s, Pellarin, and Tavares Ribeiro for the rank one case in \cite{AnglesPellarinTavares16} and later developed for arbitrary rank case by the second author and Papanikolas in \cite{GP19}. Let $\mathbb{F}_q[t]$ (resp. $A[t]$) be the ring of polynomials in $t$ with coefficients in $\mathbb{F}_q$ (resp. $A$). \textit{A Drinfeld $A[t]$-module $\psi$ of rank $r\in \ZZ_{\geq 1}$ over $\TT$} is an $\mathbb{F}_q[t]$-algebra homomorphism $\psi:A[t]\to \TT[\tau]$ given by 
\[
\psi_{\theta}:=\theta+g_1\tau+\dots+g_r\tau^r, \ \ g_r\neq 0. 
\]
For each $\psi$, there exists a unique \textit{exponential series} $\exp_{\psi}:=\sum_{i\geq 0}\gamma_i\tau^i\in \TT[[\tau]]$ satisfying $\gamma_0=1$ and for each $a\in A[t]$
\begin{equation}\label{E:FE1}
\exp_{\psi} a=\psi_a \exp_{\psi}.
\end{equation}
It induces an everywhere convergent function $\exp_{\psi}:\TT\to \TT$ given by $\exp_{\psi}(g)=\sum_{i\geq 0}\gamma_ig^{(i)}$ for any $g\in \TT$. Moreover, there exists a unique \textit{logarithm series $\log_{\psi}:=\sum_{i\geq 0}\ell_i\tau^i\in \TT[[\tau]]$ of $\psi$} which is the formal inverse of $\exp_{\psi}$ satisfying $\ell_0=1$ and for each $a\in A[t]$
\begin{equation}\label{E:FE2}
a\log_{\psi}=\log_{\psi}\psi_a.
\end{equation}
It is clear that any Drinfeld module $\phi$ over $\CC_{\infty}$ can be also considered as a Drinfeld $A[t]$-module over $\TT$ and hence one can consider the obvious extension $\exp_{\phi}:\TT\to \TT$ (resp. $\log_{\phi}: \TT \to \TT$) of the exponential function (resp. the logarithm function) of $\phi$.  For more details, we refer the reader to  \cite{GP19}.

Now we return to the theory of Drinfeld modules over $\CC_{\infty}$ and let $\phi$ be a Drinfeld module over $\CC_{\infty}$. For any $z\in \CC_{\infty}$, we define \textit{the Anderson generating function  $s_{\phi}(z,t)$ of $\phi$} by the infinite series
\[
s_{\phi}(z,t):=\sum_{i=0}^{\infty}\exp_{\phi}\Big(\frac{z}{\theta^{i+1}}\Big)t^i\in \mathbb{T}.
\]
By using the properties of the exponential series, it can be easily seen that $s_{\phi}(z,t)\in \TT$. Moreover, by \cite[Sec. 4.2]{Pel08}, we have the following series expansion 
\begin{equation}\label{E:AndPel}
s_{\phi}(z,t)=\sum_{j=0}^{\infty}\frac{\alpha_jz^{q^j}}{\theta^{q^j}-t}\in \CC_{\infty}[[t]]
\end{equation}
where $\alpha_j$ is the $j$-th coefficient of $\exp_{\phi}$. This indeed implies that as a function of $t$, $s_{\phi}(z,t)$ has poles at $t=\theta^{q^i}$ for each $i\in \ZZ_{\geq 0}$ with the residue $\Res_{t=\theta^{q^i}}s_{\phi}(z,t)=-\alpha_{i}z^{q^i}$.

The next proposition, due to Pellarin, is to derive a useful relation between $s_{\phi}(z,t)$ and certain quasi-periodic functions. Recall the rigid analytic functions $F_{\tau}^{\phi},\dots,F^{\phi}_{\tau^{r-1}}$ defined in \S1.
\begin{proposition}\cite[Sec. 4.2.2]{Pel08} \label{P:AGF} For each $1\leq k \leq r-1$, we have 
\[
F^{\phi}_{\tau^k}(z)=s_{\phi}^{(k)}(z,t)_{|t=\theta}.
\]
\end{proposition}

We finish this subsection by introducing a special element in $\TT$. We define \textit{the Anderson-Thakur element $\omega(t)$} by 
\[
\omega(t):=(-\theta)^{1/(q-1)}\prod_{i=0}^{\infty}\Big(1-\frac{t}{\theta^{q^i}}\Big)^{-1}\in \TT^{\times}.
\]
As a function of $t$, it has a pole at $t=\theta$ and furthermore we have
\begin{equation}\label{E:resomega}
-\Res_{t=\theta}\omega(t)=\theta(-\theta)^{1/(q-1)}\prod_{i=1}^{\infty}\Big(1-\theta^{1-q^i}\Big)^{-1}=\tilde{\pi}.
\end{equation}

\subsection{Partial derivatives of coefficient forms}

For any integer $r\geq 2$, recall that  $\Omega^r=\mathbb{P}^{r-1}(\CC_{\infty})\setminus \{K_{\infty}\text{-rational hyperplanes}\}$ is the Drinfeld upper half plane which can be identified as the set of elements $\zz=(z_1,\dots,z_r)^{\tr}\in \CC_{\infty}^r$ whose entries are $K_{\infty}$-linearly independent and normalized so that $z_r=1$. Throughout this section, we are mainly interested in partial derivatives $\partial_{i}$ of coefficient forms for each $1\leq i \leq r-1$.

Recall that $\phi^{\zz}$ is the Drinfeld module of rank $r$ over $\CC_{\infty}$ corresponding to the $A$-lattice of rank $r$ generated by the entries of $\zz=(z_1,\dots,z_r)^{\tr}\in \Omega^r$ over $A$ which is given by 
\[
\phi^{\zz}_{\theta}=\theta+g_1(\zz)\tau+\dots +g_{r-1}(\zz)\tau^{r-1}+g_{r}(\zz)\tau^{r}.
\]

For each $\bold{z}\in \Omega^r$, $k\in \ZZ_{\geq 1}$ and $1\leq j \leq r$, we define  
\[
\mathcal{G}^{[j]}_{\zz,k}(t):=\sum_{\substack{(a_1,\dots,a_r)\in A^r\\(a_1,\dots,a_r)\neq (0,\dots,0)}}\frac{a_j(t)}{(a_1z_1+\dots+a_rz_r)^k}\in \TT.
\]  

To ease the notation, for $k\geq 0$, we set $\Eis_k(\zz):=\Eis_{k}(A^{r}\zz)$. Recall that $\Eis_0(\zz)=-1$ and for $k\in \ZZ_{\geq 1}$, we have
	\[
	\Eis_k(\zz)=\sum_{\substack{(a_1,\dots,a_r)\in A^r\\(a_1,\dots,a_r)\neq (0,\dots,0)}}\frac{1}{(a_1z_1+\dots+a_{r-1}z_{r-1}+a_r)^k}\in \CC_{\infty}.
	\]
Then, by Lemma \ref{L:holEisenstein} and the definition of $\mathcal{G}^{[j]}_{\zz,k}(t)$, for any $k\geq 1$, we have 
\begin{equation}\label{E:partderEis}
    \partial_j(\Eis_{q^k-1})(\zz)=\mathcal{G}^{[j]}_{\zz,q^k}(t)|_{t=\theta}.
    \end{equation}

Next, we continue with investigating a certain relation between  $c_i(\zz)$ and the coefficients of the exponential and the logarithm series of $\phi$. Let $Z$ be an indeterminate over $\CC_{\infty}$. We now introduce a power series $\mathcal{T}(Z)\in \TT[[Z]]$ studied by Pellarin in \cite[Sec. 8]{Pel19} as follows. For any element $f=\sum_{i\geq 0}a_iZ^i\in \TT[[Z]]$ and $j\in \ZZ_{\geq 0}$, we let $\boldsymbol{\tau}^{j}(f):=\sum_{i\geq 0}a_i^{(j)}Z^{iq^j}$. Considering the exponential series $ \exp_{\phi}=\sum_{i\geq 0}\alpha_i(\zz)\tau^i$ of $\phi^{\zz}$, we have the following extension $\exp_{\phi^{\zz}}:\TT[[Z]]\to \TT[[Z]]$ of $\exp_{\phi^{\zz}}$, still denoted as $\exp_{\phi^{\zz}}$, and given  by $\exp_{\phi^{\zz}}(f):= \sum_{i\geq 0}\alpha_i(\zz)\boldsymbol{\tau}^{i}(f)$. We further set the coefficients $c_i(\zz)\in \TT$ so that
\[
\mathcal{T}(Z):=\sum_{i=0}^{\infty}c_i(\zz)Z^i:=\exp_{\phi^{\zz}}(Z)^{-1}\exp_{\phi^{\zz}}(Z/(\theta-t))\in \TT[[Z]].
\]

\begin{lemma}\label{L:Coef} Let  $\log_{\phi^{\zz}}=\sum_{i=0}^{\infty}\beta_i(\zz)\tau^i$ be the logarithm series of $\phi^{\zz}$. Then for each $i\geq 0$, we have 
	\[
	c_{q^i-1}(\zz)=\frac{\beta_i(\zz)}{\theta-t}+\frac{\alpha_1(\zz)\beta_{i-1}(\zz)^q}{\theta^q-t}+\dots+ \frac{\alpha_{i-1}(\zz)\beta_{1}(\zz)^{q^{i-1}}}{\theta^{q^{i-1}}-t}+\frac{\alpha_i(\zz)}{\theta^{q^i}-t}.
	\]
\end{lemma}
\begin{proof} By \eqref{E:eis1}, we obtain
	\begin{equation}\label{E:coef}
	\begin{split}
	Z\mathcal{T}(Z)&=\exp_{\phi^{\zz}}(Z/(\theta-t))\frac{Z}{\exp_{\phi^{\zz}}(Z)}\\
	&=\Big(\frac{1}{\theta-t}Z+\frac{\alpha_{1}(\zz)}{\theta^q-t}Z^q+\cdots \Big)\Big(1-\sum_{k\geq 1}\Eis_k(\zz)Z^k\Big)\\
	&=\sum_{i=0}^{\infty}c_i(\zz)Z^{i+1}.
	\end{split}
	\end{equation}	
	After comparing the coefficients of $Z^{q^i}$ in \eqref{E:coef}, we see that 
	\begin{equation}\label{E:coef2}
	c_{q^i-1}(\zz)=-\frac{\Eis_{q^i-1}(\zz)}{\theta-t}-\frac{\alpha_1(\zz)\Eis_{q^i-q}(\zz)}{\theta^q-t}-\dots-\frac{\alpha_{i-1}(\zz)\Eis_{q^i-q^{i-1}}(\zz)}{\theta^{q^{i-1}}-t}+\frac{\alpha_{i}(\zz)}{\theta^{q^i}-t}.
	\end{equation}
	On the other hand, by \eqref{E:eis2}, we know that $\Eis_{q^k-q^j}(\zz)=-\beta_{k-j}(\zz)^{q^j}$ for $k,j\geq 0$. Hence \eqref{E:coef2} becomes
	\[
	c_{q^i-1}(\zz)=\frac{\beta_i(\zz)}{\theta-t}+\frac{\alpha_1(\zz)\beta_{i-1}(\zz)^q}{\theta^q-t}+\dots+ \frac{\alpha_{i-1}(\zz)\beta_{1}(\zz)^{q^{i-1}}}{\theta^{q^{i-1}}-t}+\frac{\alpha_i(\zz)}{\theta^{q^i}-t}
	\]
	as desired.
\end{proof}	

The next proposition is fundamental to prove the main result of this subsection.
\begin{proposition}\label{P:2} For any $i\geq 1$, we have 
	\[
	(\theta^{q^i}-t)c_{q^i-1}(\zz)+g_1(\zz)^{q^{i-1}}c_{q^{i-1}-1}(\zz)+\dots+g_{i-1}(\zz)^qc_{q-1}(\zz)+g_i(\zz)c_0(\zz)=0.
	\]
	In particular, 
	\[
	g_i(\zz)=(t-\theta)((\theta^{q^i}-t)c_{q^i-1}(\zz)+g_1(\zz)^{q^{i-1}}c_{q^{i-1}-1}(\zz)+\dots+g_{i-1}(\zz)^qc_{q-1}(\zz)).
	\]
\end{proposition}
\begin{proof} Set $\mathcal{C}:=\exp_{\phi^{\zz}} \circ \frac{1}{\theta-t}\circ \log_{\phi^{\zz}}\in \TT[[\tau]].$ By Lemma \ref{L:Coef}, we have 
	\begin{equation}\label{E:coef3}
	\begin{split}
	\mathcal{C}&=\Big(\frac{1}{\theta-t}+\frac{\alpha_{1}(\zz)}{\theta^q-t}\tau+\frac{\alpha_{2}(\zz)}{\theta^{q^2}-t}\tau^2+\cdots \Big)(1+\beta_1(\zz)\tau+\beta_2(\zz)\tau^2+\cdots )\\
	&=\frac{1}{\theta-t}+\Big(\frac{\beta_1(\zz)}{\theta-t}+\frac{\alpha_{1}(\zz)}{\theta^q-t}\Big)\tau+\Big(\frac{\beta_2(\zz)}{\theta-t}+\frac{\alpha_{1}(\zz)\beta_{1}(\zz)^q}{\theta^q-t}+\frac{\alpha_{2}(\zz)}{\theta^{q^2}-t}\Big)\tau^2+\cdots+ \\
	&\ \ \ \ \ \Big(\frac{\beta_i(\zz)}{\theta-t}+\frac{\alpha_1(\zz)\beta_{i-1}(\zz)^q}{\theta^q-t}+\dots+ \frac{\alpha_{i-1}(\zz)\beta_{1}(\zz)^{q^{i-1}}}{\theta^{q^{i-1}}-t}+\frac{\alpha_i(\zz)}{\theta^{q^i}-t}\Big)\tau^i+\cdots\\
	&=\sum_{i=0}^{\infty}c_{q^i-1}(\zz)\tau^i.
	\end{split}
	\end{equation}
Since $\phi^{\zz}$ can be also considered as a Drinfeld $A[t]$-module over $\TT$,  using \eqref{E:FE2} and \eqref{E:coef3}, we obtain
\begin{equation}\label{E:coef4}
\begin{split}
1&=\mathcal{C}\circ \phi^{\zz}_{\theta-t}\\&=(c_0(\zz)+c_{q-1}(\zz)\tau+c_{q^2-1}(\zz)\tau^2+\cdots)(\theta-t+g_1(\zz)\tau+g_2(\zz)\tau^2+\dots+g_r(\zz)\tau^r)\\
&=(\theta-t)c_0(\zz)+((\theta^q-t)c_{q-1}(\zz)+g_1(\zz)c_0(\zz))\tau+\cdots+ \\
&\ \ \ \ ((\theta^{q^i}-t)c_{q^i-1}(\zz)+g_1(\zz)^{q^{i-1}}c_{q^{i-1}-1}(\zz)+\dots+g_{i-1}(\zz)^qc_{q-1}(\zz)+g_i(\zz)c_0(\zz))\tau^i+\cdots.
\end{split}
\end{equation}
Thus the first part of the proposition follows from comparing $\tau^i$-th coefficients in \eqref{E:coef4}. The second part of the proposition is a consequence of the first part and the fact that $c_0(\zz)=(\theta-t)^{-1}$ which can be easily deduced from \eqref{E:coef4}.
\end{proof}

\begin{remark} Let $\mathcal{C}$ be as in the proof of Proposition \ref{P:2}. Using \eqref{E:FE1}, it is clear that $\phi^{\zz}_{\theta-t}\circ \mathcal{C}=1$. Moreover, by using \eqref{E:coef3}, for each $i\geq 1$, one can obtain
	\begin{equation}\label{E:P}
	(\theta-t)c_{q^i-1}(\zz)+g_1(\zz)c_{q^{i-1}-1}(\zz)^{(1)}+\dots+g_{i-1}(\zz)c_1(\zz)^{(i-1)}+g_i(\zz)c_0(\zz)^{(i)}=0.
	\end{equation}
	We remark that one can provide another proof for \cite[Thm. 3.2]{CG21} by using the identity \eqref{E:P} and \cite[Prop. 3.18]{CG21}. We leave the details to the reader.
\end{remark}

Following the notation in \cite[Sec. 3]{CG21}, for each $1\leq i \leq r$, we set $s_i(\zz;t)$ to be the Anderson generating function  
$
s_i(\zz,t):=s_{\phi^{\zz}}(z_i,t)
$
and define the matrix
\begin{equation}\label{D:F}
\mathcal{F}(\zz,t):=\begin{pmatrix}
s_1(\zz,t)&\dots &s_1^{(r-1)}(\zz,t)\\
\vdots & & \vdots \\
s_r(\zz,t)&\dots &s_r^{(r-1)}(\zz,t)
\end{pmatrix}\in \Mat_{r}(\TT).
\end{equation}
By \cite[Prop. 3.4]{CG21}, $\mathcal{F}(\zz,t)$ is indeed an element in $\GL_{r}(\TT)$. For each $1\leq i,j\leq r$, we set $\tilde{L}_{ij}(\zz,t)$ to be the $(i,j)$-cofactor of $\mathcal{F}(\zz,t)$. 

Recall from \S3, the $h$-function of Gekeler $h_r:\Omega^r\to \mathbb{C}_{\infty}$.

\begin{lemma} We have 
	\begin{equation}\label{E:mat}
	\begin{split}
	&\begin{pmatrix}
	\mathcal{G}^{[1]}_{\zz,1}(t)&\dots &\mathcal{G}^{[r]}_{\zz,1}(t)\\
	\mathcal{G}^{[1]}_{\zz,q}(t)&\dots &\mathcal{G}^{[r]}_{\zz,q}(t)\\
	\vdots& &\vdots\\
	\mathcal{G}^{[1]}_{\zz,q^{r-1}}(t)&\dots &\mathcal{G}^{[r]}_{\zz,q^{r-1}}(t)
	\end{pmatrix}\\
	&=-\frac{\tilde{\pi}^{\frac{q^r-1}{q-1}}h_r(\zz)}{\omega(t)}\begin{pmatrix}
	c_0(\zz)&  & &\\
	c_{q-1}(\zz)&c_0(\zz)^{(1)} & & \\
	\vdots& &\ddots & \\
	c_{q^{r-1}-1}(\zz)&c_{q^{r-2}-1}(\zz)^{(1)} & \dots & c_0(\zz)^{(r-1)}
	\end{pmatrix}\begin{pmatrix}
	\tilde{L}_{11}(\zz,t)&\dots &\tilde{L}_{r1}(\zz,t)\\
	\tilde{L}_{12}(\zz,t)&\dots &\tilde{L}_{r2}(\zz,t)\\
	\vdots& &\vdots\\
	\tilde{L}_{1r}(\zz,t)&\dots &\tilde{L}_{rr}(\zz,t)
	\end{pmatrix}.
	\end{split}
	\end{equation}
\end{lemma}
\begin{proof} Using Pellarin's result \cite[Prop. 4.8.3 and Thm. 4.8.8]{Pel19} (see also \cite[Thm. 3.11]{CG21}), for any $0\leq j \leq r-1$, one can have 	\begin{equation}\label{E:Comp1}
	(\mathcal{G}^{[1]}_{\zz,q^j}(t),\dots,\mathcal{G}^{[r]}_{\zz,q^j}(t))\begin{pmatrix}
	s_1(\zz,t)\\
	\vdots \\
	\vdots \\
	s_r(\zz,t)
	\end{pmatrix}=-c_{q^j-1}(\zz).
	\end{equation}
	On the other hand, by \cite[Prop. 3.18]{CG21}, we have 
	\begin{equation}\label{E:Comp2}
	(\mathcal{G}^{[1]}_{\zz,1}(t),\dots,\mathcal{G}^{[r]}_{\zz,1}(t))\mathcal{F}(\zz,t)=-(c_0(\zz),0,\dots,0).
	\end{equation}
	Thus, using \eqref{E:Comp1} and \eqref{E:Comp2} as well as applying suitable twisting operator, we obtain
	\[
	\begin{pmatrix}
	\mathcal{G}^{[1]}_{\zz,1}(t)&\dots &\mathcal{G}^{[r]}_{\zz,1}(t)\\
	\mathcal{G}^{[1]}_{\zz,q}(t)&\dots &\mathcal{G}^{[r]}_{\zz,q}(t)\\
	\vdots& &\vdots\\
	\mathcal{G}^{[1]}_{\zz,q^{r-1}}(t)&\dots &\mathcal{G}^{[r]}_{\zz,q^{r-1}}(t)
	\end{pmatrix}\mathcal{F}(\zz,t)=-\begin{pmatrix}
	c_0(\zz)&  & &\\
	c_{q-1}(\zz)&c_0(\zz)^{(1)} & & \\
	\vdots& &\ddots & \\
	c_{q^{r-1}-1}(\zz)&c_{q^{r-2}-1}(\zz)^{(1)} & \dots & c_0(\zz)^{(r-1)}
	\end{pmatrix}.
	\]
	Now the proposition follows from the right multiplication of both sides of above by the inverse of $\mathcal{F}(\zz,t)$ and using \cite[Prop. 3.4]{CG21}.
\end{proof}
For each $1 \leq i\leq r-1$ and $1\leq j \leq r$, recall the function $L_{ij}:\Omega^r\to \CC_{\infty}$ which sends $\zz\in \Omega^r$ to the $(i,j)$-cofactor of the period matrix $P_{\zz}$. It can be easily seen by using Proposition \ref{P:AGF} that $L_{ij}$ is a rigid analytic function on $\Omega^r$.

\begin{theorem}\label{T:GM2} For $1\leq i,j \leq r-1$, we have
		\[\partial_{j}(g_i)(\zz)=E_r^{[j]}(\zz)g_i(\zz)+\tilde{\pi}^{q+\dots+q^{r-1}}h_r(\zz)L_{j(i+1)}(\zz).
        \]
\end{theorem}
\begin{proof}  
	Set 
	\[
	G(\zz,t):=(\theta^{q^i}-t)\mathcal{G}^{[j]}_{\zz,q^i}(t)+\sum_{k=1}^{i-1}\mathcal{G}^{[j]}_{\zz,q^k}(t)g_{i-k}(\zz)^{q^k}.
	\]
	 Observe from \eqref{E:mat} that 
	\begin{align*}
	G(\zz,t)&=-\frac{\tilde{\pi}^{\frac{q^r-1}{q-1}}h_r(\zz)}{\omega(t)}\Big((\theta^{q^i}-t)(c_{q^i-1}(\zz)\tilde{L}_{j1}(\zz,t)+c_{q^{i-1}-1}(\zz)^{(1)}\tilde{L}_{j2}(\zz,t)\\
	&\ \ \ +\dots+c_{q-1}(\zz)^{(i-1)}\tilde{L}_{ji}(\zz,t)) +\tilde{L}_{j(i+1)}(\zz,t)+g_{i-1}(\zz)^qc_{q-1}(\zz)\tilde{L}_{j1}(\zz,t)\\
	&\ \ \ +g_{i-1}(\zz)^qc_0(\zz)^{(1)}\tilde{L}_{j2}(\zz,t)+g_{i-2}(\zz)^{q^2}c_{q^2-1}(\zz)\tilde{L}_{j1}(\zz,t) + g_{i-2}(\zz,t)^{q^2}c_{q-1}(\zz)^{(1)}\tilde{L}_{j2}(\zz,t) \\
	&\ \ \ +g_{i-2}(\zz)^{q^2}c_0(\zz)^{(2)}\tilde{L}_{j3}(\zz,t)+\cdots  + g_1(\zz)^{q^{i-1}}c_{q^{i-1}-1}(\zz)\tilde{L}_{j1}(\zz,t)\\
	&\ \ \  +g_1(\zz)^{q^{i-1}}c_{q^{i-2}-1}(\zz)^{(1)}\tilde{L}_{j2}(\zz,t)+\dots +g_1(\zz)^{q^{i-1}}c_0(\zz)^{(i-1)}(\zz)\tilde{L}_{ji}(\zz,t)\Big)\\
	&=-\frac{\tilde{\pi}^{\frac{q^r-1}{q-1}}h_r(\zz)}{\omega(t)}\Big((\theta^{q^i}-t)c_{q^i-1}(\zz)+g_1(\zz)^{q^{i-1}}c_{q^{i-1}-1}(\zz)+\dots +g_{i-1}(\zz)^qc_{q-1}(\zz)\Big)\tilde{L}_{j1}(\zz,t)\\
	&\ \ \  -\frac{\tilde{\pi}^{\frac{q^r-1}{q-1}}h_r(\zz)}{\omega(t)}\Big((\theta^{q^{i-1}}-t)c_{q^{i-1}-1}(\zz)+g_1(\zz)^{q^{i-2}}c_{q^{i-2}-1}(\zz)+\cdots+g_{i-1}(\zz)c_{0}(\zz)\Big)^{(1)}\tilde{L}_{j2}(\zz,t) \\
	&\ \ \   -\cdots \\
	&\ \ \  -\frac{\tilde{\pi}^{\frac{q^r-1}{q-1}}h_r(\zz)}{\omega(t)}\Big((\theta^{q}-t)c_{q-1}(\zz)+g_1(\zz)c_{0}(\zz)\Big)^{(i-1)}\tilde{L}_{ji}(\zz,t)-\frac{\tilde{\pi}^{\frac{q^r-1}{q-1}}h_r(\zz)}{\omega(t)}\tilde{L}_{j(i+1)}(\zz,t)\\
	&=-\frac{\tilde{\pi}^{\frac{q^r-1}{q-1}}h_r(\zz)}{\omega(t)}\Big(\tilde{L}_{j(i+1)}(\zz,t)+\frac{g_i(\zz)}{t-\theta}\tilde{L}_{j1}(\zz,t)\Big)
	\end{align*}
	where the last equality follows from Proposition \ref{P:2}. 
	On the other hand, by \eqref{E:EisCoef}, we have 
	\begin{equation}\label{E:gExp}
	g_i(\zz)=(\theta^{q^i}-\theta)\Eis_{q^i-1}(\zz)+\sum_{k=1}^{i-1}\Eis_{q^k-1}(\zz)g_{i-k}(\zz)^{q^k}.
	\end{equation}
	Applying the operator $\partial_{j}$ to both sides of \eqref{E:gExp} and using \eqref{E:partderEis}, we obtain 
	\begin{equation}\label{E:det0}
	\partial_j(g_i)(\zz)=G(\zz,\theta).
	\end{equation}
Moreover, by \cite[Thm. 5.4(i), Cor. 5.9]{CG21} and Proposition \ref{P:AGF}, we have 
	\[
	E_r^{[j]}(\zz)=-\frac{\tilde{\pi}^{\frac{q^r-1}{q-1}}h_r(\zz)}{\omega(t)(t-\theta)}\tilde{L}_{j1}(\zz,t)_{|t=\theta}
	\]
	and 
	\[
	\Res_{t=\theta}\tilde{L}_{j(i+1)}(\zz,t)=L_{j(i+1)}(\zz).
	\]
	Thus by \eqref{E:resomega}, \eqref{E:det0} and the above calculation, we have 
	\begin{align*}
	\partial_j(g_i)(\zz)&=G(\zz,\theta)\\
	&=-\frac{\tilde{\pi}^{\frac{q^r-1}{q-1}}h_r(\zz)}{\omega(t)}\Big(\tilde{L}_{j(i+1)}(\zz)+\frac{g_i(\zz)}{t-\theta}\tilde{L}_{j1}(\zz)\Big)_{|t=\theta}\\
	&=E_r^{[j]}(\zz)g_i(\zz)+\tilde{\pi}^{q+\dots+q^{r-1}}h_r(\zz)L_{j(i+1)}(\zz)
	\end{align*}
	as desired.
\end{proof}
For any $\gamma \in \GL_r(A)$ and $\zz\in \Omega^r$, recall the coefficients $\mathfrak{c}^{\gamma}_{jl}(\zz)$ from \S4.1. 
\begin{corollary} For any $\gamma\in \GL_{r}(A)$, $\zz\in \Omega^r$ and $1\leq i,j\leq r-1$, we have 
\[
(h_rL_{j(i+1)})(\gamma\cdot\zz)=j(\gamma,\zz)^{q^i}\det(\gamma)^{-1}\Big(\sum_{l=1}^{r-1}(h_rL_{l(i+1)})(\zz)\mathfrak{c}^{\gamma}_{jl}(\zz)\Big).
\]
\end{corollary}
\begin{proof}
By Theorem \ref{T:GM2}, we see that 
\[
(h_rL_{j(i+1)})(\zz)=\tilde{\pi}^{-(q+\dots+q^{r-1})}\left(\partial_j(g_i(\zz))-E_r^{[j]}(\zz)g_i(\zz)\right)=\tilde{\pi}^{-(q+\dots+q^{r-1})}\mathcal{D}_{j,q^i-1}(g_i)(\zz).
\]
Hence the corollary follows immediately from Lemma \ref{L:1}(iii).
\end{proof}
We now continue with our next proposition which leads an alternative description for Gekeler's $h$-function $h_r$.
\begin{proposition}\label{P:hfunc} Consider the $(r-1)\times (r-1)$ matrix $\mathcal{R}$ given by 
	\[
 \mathcal{R}:=\begin{pmatrix}
	L_{12}&\cdots & \cdots &L_{1r}\\
	\vdots & & &\vdots  \\
	\vdots & & & \vdots \\
	L_{(r-1)2} &\cdots &\cdots & L_{(r-1)r}
	\end{pmatrix}.
	\]    
    Then, we have
\begin{equation}\label{E:DET}
	\det((h_r\mathcal{R})(\zz))=\tilde{\pi}^{-(q+\dots+q^{r-1})(r-2)}h_r(\zz).
	\end{equation}
    In particular, we have 
 	\[
	\D_{(q-1,\dots,q^{r-1}-1)}(g_1,\dots,g_{r-1})(\zz)=\tilde{\pi}^{q+\dots+q^{r-1}}h_r(\zz).
	\]
\end{proposition}
\begin{proof} Writing the inverse of the period matrix $P_{\zz}$ in terms of its cofactors, we have $P_{\zz}^{-1}=\det(P_{\zz})^{-1}(L_{ji})_{i,j}$. Since the $(r,1)$-entry of $P_{\zz}=(P_{\zz}^{-1})^{-1}$ is equal to the $(1,r)$-cofactor of $P_{\zz}^{-1}$ times $\det(P^{-1}_\zz)^{-1}$, we obtain 
\[
	-z_r=-1=(-1)^{r+1}\frac{\det(\mathcal{R})(\zz)}{ \det(P_{\zz})^{r-2}}.
	\]
    On the other hand, \cite[(3.4)]{CG21} and Proposition \ref{P:AGF} imply that $\det(P_{\zz})=-\tilde{\pi}^{-(q+\dots+q^{r-1})}h_r(\zz)^{-1}$. Since $z_r=1$, we thus obtain 
	\[
	\det((h_r\mathcal{R})(\zz))=\tilde{\pi}^{-(q+\dots+q^{r-1})(r-2)}h_r(\zz)
    \]
which is the first assertion. For the second assertion, note, by Theorem \ref{T:GM2}, that 
\[
\D_{(q-1,\dots,q^{r-1}-1)}(g_1,\dots,g_{r-1})(\zz)=\tilde{\pi}^{(q+\dots+q^{r-1})(r-1)}\det((h_r\mathcal{R})(\zz)).
\]
Then the desired statement simply follows from the first assertion.
    \end{proof}

\begin{remark} As observed by the referee, we indeed have
\begin{multline*}
\tilde{\pi}^{(q+\dots+q^{r-1})(r-1)}\det((h_r\mathcal{R})(\zz))=\det\left(\left(g_r(\zz)\partial_j\left(\frac{g_i}{g_r}\right)(\zz)\right)_{1\leq i,j\leq r-1}\right)\\=\frac{1}{g_r(\zz)}\det\begin{pmatrix}
      \partial_1(g_1)(\zz)&\cdots & \partial_{r-1}(g_{1})(\zz)&g_1(\zz)\\
      \vdots &  & \vdots & \vdots \\
      \partial_1(g_r)(\zz) &\cdots & \partial_{r-1}(g_{r})(\zz)&g_r(\zz)
  \end{pmatrix}.
\end{multline*}
Here the first equality follows from Theorem \ref{T:GM2} and the second equality follows from an expansion argument used in the proof of Lemma \ref{L:det}.
\end{remark}

\subsection{The $\CC_{\infty}$-algebra $\mathcal{M}_r$} 

In this subsection, our goal is to analyze certain properties of the $\CC_{\infty}$-algebra $\mathcal{M}_r$ introduced in \S1 which could be seen as an initial step to investigate differential structure of Drinfeld modular forms of arbitrary rank. Recall that we define the $\CC_{\infty}$-algebra $\mathcal{M}_r$ generated by the following rigid analytic functions 
\[
\mathcal{M}_r:=\mathbb{C}_\infty[g_1,\dots,g_{r-1},h_r, \partial_j(g_i), E_r^{[j]}\mid 1\leq i,j\leq r-1].
\]

The following proposition was implicitly used in an earlier version of the paper and its statement as well as its proof become more complete in the present version thanks to the suggestions and comments provided by Ernst-Ulrich Gekeler.
\begin{proposition}\label{P:Gek} For any $1\leq i,j \leq r-1$, let $\mathcal{D}:=\partial_{i}\circ\partial_{j}$ be a homogeneous linear differential operator of degree 2.  For $k\in \ZZ_{\geq 1}$ and any non-constant $a\in A$, if $f$ is either a coefficient form $g_{k}$ or an Eisenstein series of weight $q^k-1$, then $\mathcal{D}(f)\equiv0$. In particular, for any $1\leq i,j \leq r-1$, we have 
	\[
	\partial_{j}(E_r^{[i]})(\zz)=-E_r^{[j]}(\zz)E_r^{[i]}(\zz).
	\]\end{proposition}
\begin{proof} Recall the notation used in \S2.1 and \S3. Suppose first that $f(\zz)=\Eis_{q^k-1}(A^r\zz)$. 
Then,  we have
	\[
	\partial_{j}(f)(\zz)=\sum_{\substack{(a_1,\dots,a_r)\in A^r\\(a_1,\dots,a_r)\neq (0,\dots,0)}}\frac{a_j}{(a_1z_1+\dots+a_rz_r)^{q^k}}
	\]
and by Lemma \ref{L:holEisenstein}, the infinite sum on the right hand side above converges. Since the characteristic of $\mathbb{C}_{\infty}$ is $p$, the above identity implies that the derivation $\partial_{i}$ vanishes on $\partial_{j}f$ and it finishes the proof. On the other hand, if $f=g_{k}$, then the proposition follows from applying the previous argument and considering \eqref{E:EisCoef} which provides a recursive formula for the coefficient forms in terms of Eisenstein series of weight $q^{\mu}-1$ for positive integers $\mu$ and $q$-th power of the lower forms. Finally, noting the definition of the functions $E_r^{[1]},\dots,E_r^{[r-1]}$, one can see that the last assertion is an immediate consequence of the first statement.
\end{proof}

Now we are ready to prove the main result of this subsection. 
\begin{theorem}\label{Thm:QM} The $\CC_{\infty}$-algebra $\mathcal{M}_r$ is stable under the partial derivatives and strictly contains the graded $\CC_{\infty}$-algebra of Drinfeld modular forms for $\GL_{r}(A)$.
\end{theorem}
\begin{proof} Since the $\mathbb{C}_{\infty}$-algebra of Drinfeld modular forms for $\GL_r(A)$ is generated by $g_1,\dots,g_{r-1}$ and $h_r$ \cite[Thm. 17.5]{BBP18}, $\mathcal{M}_r$ strictly contains the graded $\CC_{\infty}$-algebra of Drinfeld modular forms for $\GL_{r}(A)$. On the other hand, note that 

	\[
	E_r^{[j]}(\zz)=\frac{\partial_{j}(g_{r})(\zz)}{g_{r}(\zz)}=-\frac{h_r(\zz)^{q-2}\partial_{j}(h_r)(\zz)}{h_r(\zz)^{q-1}}=-\frac{\partial_{j}(h_r)(\zz)}{h_r(\zz)}		
	\] 
	and hence we have
    \begin{equation}\label{E:hder}
        \partial_{j}(h_r)(\zz)=-h_r(\zz)E_r^{[j]}(\zz).
        \end{equation}
    Then the stability of $\mathcal{M}_r$ under partial derivatives follows from \eqref{E:hder} and Proposition \ref{P:Gek}.
\end{proof}
We now define the $\CC_{\infty}$-algebra $\widetilde{\mathcal{M}_r}$ generated by the following rigid analytic functions 
\[
\widetilde{\mathcal{M}_r}=\mathbb{C}_\infty[g_i,h_rL_{ij}\mid 1\leq i\leq r-1,~1\leq j\leq r].
\]

\begin{theorem}\label{T:eq} We have
\[
    \widetilde{\mathcal{M}_r}=\mathcal{M}_r.
    \]
    \end{theorem}
\begin{proof} We first show that $\mathcal{M}_r\subseteq \widetilde{\mathcal{M}_r}$. 
    Using \cite[Thm. 1.3, Thm. 5.5, Cor. 5.9]{CG21}, for each $1\leq i \leq r-1$, we have 
\begin{equation}\label{E:falseeis}
E_r^{[i]}=\tilde{\pi}^{q+\cdots+q^{r-1}}h_rL_{i1}
\end{equation} 
implying that $E_{r}^{[i]}\in \widetilde{\mathcal{M}_r}$. Thus, using Theorem \ref{T:GM2}, we further see that $\partial_{\ell}(g_i)\in \widetilde{\mathcal{M}_r}$ for $1\leq \ell \leq r-1$. Finally, Proposition \ref{P:hfunc} implies that $h_r\in \widetilde{\mathcal{M}_r}$ and hence $\mathcal{M}_r\subseteq \widetilde{\mathcal{M}_r}$. To show that $\widetilde{\mathcal{M}_r}\subseteq \mathcal{M}_r$, it is enough to observe that $h_rL_{ij}\in \mathcal{M}_r$ for $1\leq j \leq r$ and it follows from Theorem \ref{T:GM2} and \eqref{E:falseeis}. Hence $\widetilde{\mathcal{M}_r}=\mathcal{M}_r$ as desired. 
\end{proof}

We finally record a lemma on how to calculate the partial derivatives of the functions $h_rL_{ij}$.
\begin{lemma}\label{L:D1} Let $\zz\in \Omega^r$. For each $1\leq i,j,l \leq r-1$, we have 
	\[
	\partial_{l}(h_rL_{j(i+1)})(\zz)=-E_r^{[j]}(\zz)(h_rL_{l(i+1)})(\zz).
	\]
\end{lemma}
\begin{proof}  Observe that we have 
	\begin{align*}
	\partial_{l}(\tilde{\pi}^{q+\cdots+q^{r-1}} h_rL_{j(i+1)})(\zz)&=(\partial_l\circ\partial_{j})(g_i)(\zz)-\partial_{l}(E_r^{[j]}g_i)(\zz)\\
	&=E_r^{[l]}(\zz)E_r^{[j]}(\zz)g_i(\zz)-E_r^{[j]}(\zz)E_r^{[l]}(\zz)g_i(\zz)\\
	&\ \ \ \ \ \ \ -\tilde{\pi}^{q+\cdots+q^{r-1}}E_r^{[j]}(\zz)(h_rL_{l(i+1)})(\zz)\\
	&=-\tilde{\pi}^{q+\cdots+q^{r-1}}E_r^{[j]}(\zz)(h_rL_{l(i+1)})(\zz)
	\end{align*}
	where the first equality follows from Theorem \ref{T:GM2} and the second equality follows from Proposition \ref{P:Gek}. This calculation finishes the proof.
\end{proof}

\end{document}